\pdfoutput=1 

\documentclass[12pt]{extarticle}
\usepackage[utf8]{inputenc}
\usepackage[T1]{fontenc}

\usepackage{caption}
\oddsidemargin \evensidemargin
\usepackage{float}
\usepackage{resizegather}
\usepackage{amssymb}
\usepackage{pmboxdraw}
\usepackage{amsmath}
\usepackage{bbm}
\usepackage{amsthm}
\usepackage{amssymb}
\usepackage{dsfont}
\usepackage[dvipsnames]{xcolor}
\usepackage{graphicx}
\usepackage{alphabeta}
\usepackage{tikz,pgfplots}
\pgfplotsset{compat=1.18}
\usepackage{pgfplots}
\usepackage{comment}
\usepackage[margin=0.5cm]{caption}
\usepackage{hyperref}
\usepackage{subcaption}
\usepackage[all]{xy}
\usepackage[greek,english]{babel}
\usepackage[export]{adjustbox}
\usepackage{hyperref}
\usepackage{mathrsfs}  
\hypersetup{
    colorlinks=true,
    linkcolor=myblue,
    filecolor=magenta,      
    urlcolor=myblue,
    citecolor=BurntOrange,
}
\usepackage{enumitem}
\usepackage{booktabs}

\definecolor{myblue}{rgb}{0.25,0.45,0.99}
\definecolor{myorange}{rgb}{0.8500, 0.3250, 0.0980}
\definecolor{myyellow}{rgb}{0.9290, 0.6940, 0.1250}
\definecolor{mypurple}{rgb}{0.4940, 0.1840, 0.5560}
\definecolor{mygreen}{rgb}{0.4660, 0.6740, 0.1880}

\usepackage[activate={true,nocompatibility},final,tracking=true,kerning=true,spacing=true,factor=1100,stretch=10,shrink=10]{microtype}

\let\OLDthebibliography\thebibliography
\renewcommand\thebibliography[1]{
  \OLDthebibliography{#1}
  \setlength{\parskip}{0.4pt}
  \setlength{\itemsep}{2.8pt plus 0.0ex}
}

\usepackage{listings}

\usepackage{algorithm}
\usepackage{algpseudocode}

\newtheorem{theorem}{Theorem}[section]
\newtheorem{theorem*}[theorem]{Theorem*}

\newtheorem{lemma}[theorem]{Lemma}

\newtheorem{proposition}[theorem]{Proposition}

\newtheorem{thm/conj}[theorem]{Theorem/Conjecture}

\theoremstyle{definition}
\newtheorem{examplex}[theorem]{Example}
\newenvironment{example}
{\pushQED{\qed}\examplex}
{\popQED\endexamplex}

\newtheorem{definition}[theorem]{Definition}
\newtheorem{remark}[theorem]{Remark}

\theoremstyle{remark}

\newcommand{\C}{\mathbb{C}}

\newcommand\restr[2]{{
  \left.\kern-\nulldelimiterspace 
  #1
  \vphantom{\big|} 
  \right|_{#2}
  }}

\usepackage{titling} 
\predate{}
\postdate{}
\date{}
\usepackage{cleveref}

\usepackage[frozencache,cachedir=.]{minted}
\setminted[julia]{frame=lines,
rulecolor=\color{white!80!black},
fontsize=\small,
numbers=right,
numbersep=-5pt,
obeytabs=true,
tabsize=4}
\usepackage{textgreek}

\title{\textbf{Lissajous Varieties}
}
\author{Francesco Maria Mascarin and Simon Telen }

\begin{document}

\maketitle

\begin{abstract}
    \noindent This paper studies affine algebraic varieties parametrized by sine and cosine functions, generalizing algebraic Lissajous figures in the plane. We show that, up to a combinatorial factor, the degree of these varieties equals the volume of a polytope. We deduce defining equations from rank constraints on a matrix with polynomial entries. We discuss applications of Lissajous varieties in dynamical systems, in particular the Kuramoto model. This leads us to study connections with convex optimization and Lissajous~discriminants. 
\end{abstract}

\section{Introduction}
A matrix $A \in \mathbb{Q}^{d \times n}$ defines a linear space ${\rm Row}(A) \subseteq \mathbb{C}^n$ by taking the $\mathbb{C}$-linear span of its rows. In this paper, we are interested in the algebraic varieties obtained by taking the coordinate-wise cosine of points in translates of ${\rm Row}(A)$. Concretely, let us define the map 
\begin{equation} \label{eq:cos} {\rm cos}: \mathbb{C}^n \longrightarrow \mathbb{C}^n, \quad (x_1, \ldots, x_n) \, \longmapsto \, (\cos(x_1), \ldots, \cos(x_n)).\end{equation}
For any vector $b \in \mathbb{C}^n$, we let ${\cal L}_{A,b}$ be the image of the affine-linear space $L_{A,b} = {\rm Row}(A) -  \frac{b\pi}{2} = \{ x -  \frac{b \pi}{2}\, : \, x \in {\rm Row}(A)\} \subseteq \mathbb{C}^n$ under the map ${\rm cos}$. In symbols, we set 
\begin{equation} \label{eq:XAb} {\cal L}_{A,b} \, = \, {\rm cos}(L_{A,b}).\end{equation}
Notice that ${\cal L}_{A,{\bf 0}} = {\rm cos}({\rm Row}(A))$ and ${\cal L}_{A,{\bf 1}} = {\rm sin}({\rm Row}(A))$, where ${\bf 0} \in \mathbb{C}^n$ and ${\bf 1} \in \mathbb{C}^n$ are the all-zeros and the all-ones vector respectively, and ${\rm sin}$ is the coordinate-wise sine map, analogous to \eqref{eq:cos}. It is convenient to write ${\cal C}_A = {\cal L}_{A,{\bf 0}}$ and ${\cal S}_A = {\cal L}_{A,{\bf 1}}$ for these special cases. 

The requirement that $A$ has rational entries ensures that ${\cal L}_{A,b}$ is an irreducible affine variety of dimension ${\rm rank}(A)$. In particular, the set $\cos(L_{A,b}) \subseteq \mathbb{C}^n$ is Zariski closed in $\mathbb{C}^n$; see Lemma \ref{lem:noclosure}. Our first goal is to determine its degree and defining equations. We present some familiar examples and illustrate some features of ${\cal L}_{A,b}$. 

\begin{example}
The plane curves obtained from $A \in \mathbb{Q}^{1 \times 2}$ and $b \in \mathbb{C}^2$ are known in the literature as \emph{Lissajous curves}, which motivates the title of our paper. Such curves describe two objects driven in simple harmonic motion along the $x$- and $y$-axis \cite{greenslade1993all}. In that context, one usually allows real entries for $A$, in which case ${\cal L}_{A,b}$ is not necessarily an algebraic curve. Moreover, one restricts to real vectors $b \in \mathbb{R}^2$ and focuses on the real points of ${\cal L}_{A,b}$. 
\end{example}

\begin{example} \label{ex:circleintro} The real points of $X = \{ (x,y) \in \mathbb{C}^2 \, : \, x^2 + y^2 = 1 \}$ form a circle of radius one, centered at the origin. The curve $X$ is parametrized by $(\cos(t), \sin(t)) = (\cos(t),\cos(t-\frac{\pi}{2}))$. It is the Lissajous variety $X = {\cal L}_{A,b}$ with $A = \begin{pmatrix}
1 & 1 
\end{pmatrix} \in \mathbb{Q}^{1 \times 2}$ and $b = (0, 1)$.
\end{example}

\begin{example} \label{ex:elliptopeintro}
Let $A = \left ( \begin{smallmatrix} 1 & 0 & -1 \\ -1 & 1 & 0 \end{smallmatrix} \right)$. The surfaces ${\cal C}_A$ and ${\cal S}_A$ are parametrized as follows: 
\[ {\cal C}_A \, : \, (\cos(t_1-t_2), \, \cos(t_2),\,  \cos(t_1) ), \quad {\cal S}_A \,: \, (\sin(t_1-t_2), \,  \sin(t_2), \,  -\sin(t_1)). \]
Figure \ref{fig:C3} shows these surfaces. 
\begin{figure}
\centering
\includegraphics[height = 5cm]{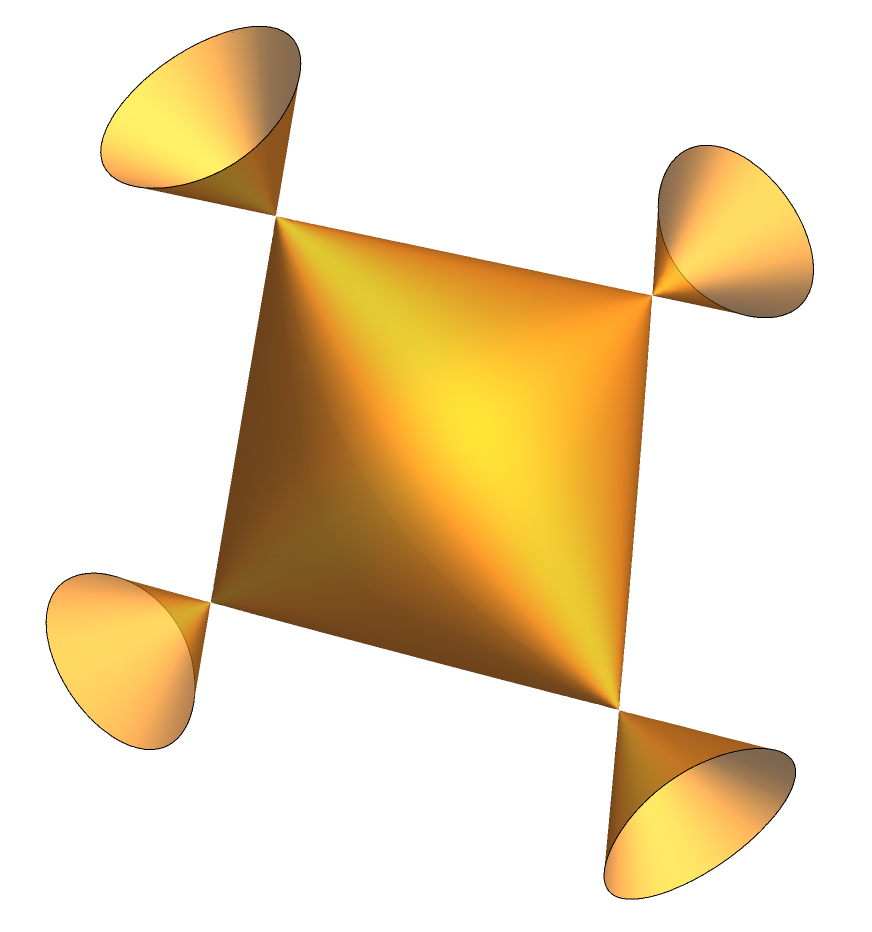} \quad  \quad \quad \quad 
\includegraphics[height = 5cm]{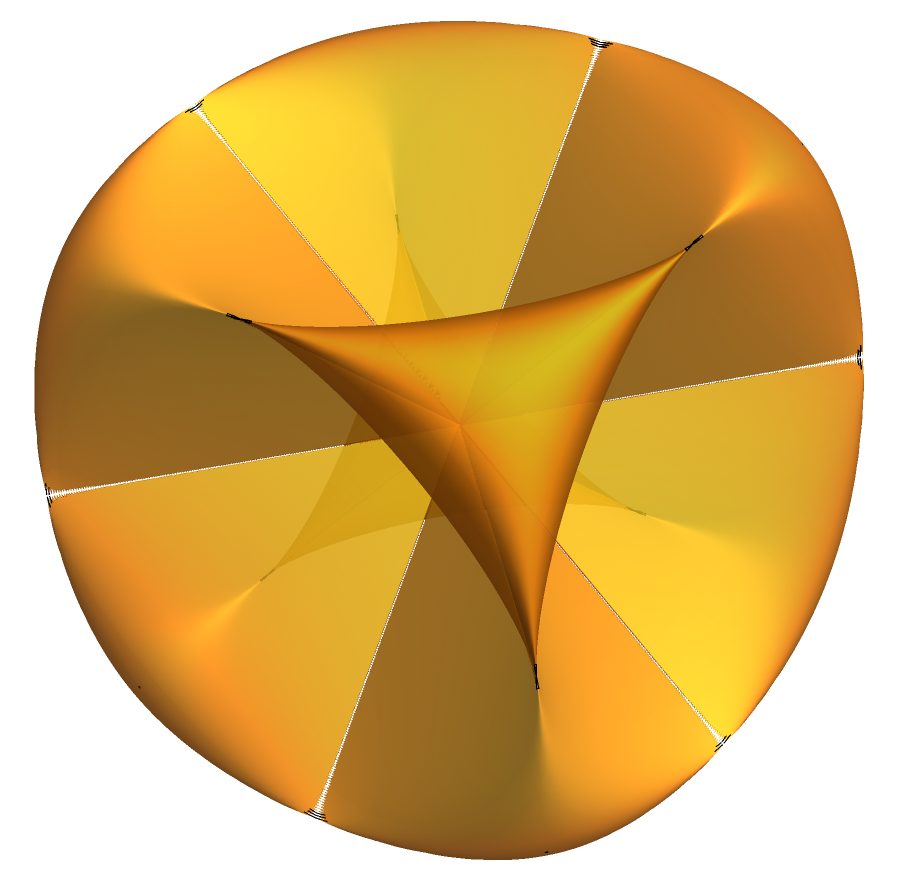}
\caption{The surfaces ${\cal C}_A$ (left) and ${\cal S}_A$ (right) from Example \ref{ex:elliptopeintro}.}
\label{fig:C3}
\end{figure}
We compute that ${\cal C}_A$ is given by $1 + 2\, xyz - x^2 - y^2 - z^2 = 0$. This is Cayley's cubic surface with four nodes. It is the algebraic boundary of the elliptope
\[ {\cal E}_3 \, = \, \left \{ (x,y,z) \in \mathbb{R}^3 \, : \, \begin{pmatrix}
1 & x & y \\ x & 1 & z \\ y & z & 1
\end{pmatrix}  \text{ is positive semi-definite}  \right \}, \]
a standard example of a spectrahedron in semi-definite programming \cite{vinzant2014spectrahedron}. The surface ${\cal S}_A$ is 
\begin{equation} \label{eq:SA} x^4 + 4 \, x^2y^2z^2 - 2 \,  x^2 y^2  - 2 \, x^2 z^2+ y^4 - 2 \, y^2z^2 + 
 z^4 \, = \, 0. \end{equation}
It has degree six. As the parameter $b$ varies, the Lissajous varieties ${\cal L}_{A,b}$ form a family of surfaces. A generic fiber is reduced and irreducible of degree six. The special fiber for $b = {\bf 0}$ is the cubic surface ${\cal C}_A$ with multiplicity two. This is a general phenomenon, see Section \ref{sec:2}. 
\end{example}

In the next few paragraphs we summarize our main contributions and sketch the outline of the paper. Since ${\cal L}_{A,b}$ only depends on the affine space $L_{A,b}$, one may replace $A$ by a matrix with the same row span without altering ${\cal L}_{A,b}$. Therefore, it is not restrictive to assume that $A$ is of rank $d$, and we may clear denominators so that $A$ has integer entries. Additionally, we make the following technical but equally non-restrictive assumption. The lattice $\mathbb{Z}^d$ is an abelian group with entry-wise addition. Let $\mathbb{Z}A \subseteq \mathbb{Z}^d$ be the subgroup generated by the columns $a_1, \ldots, a_n \in \mathbb{Z}^d$ of $A$. We assume that $\mathbb{Z}A = \mathbb{Z}^d$. The following statement, proved and stated in more detail in Section \ref{sec:2}, gives the degree of ${\cal L}_{A,b}$, i.e., the number of intersection points of ${\cal L}_{A,b}$ with a generic affine-linear space of complementary dimension $n- \dim({\cal L}_{A,b})$. 

\begin{theorem} \label{thm:mainintro}
Let $A = \begin{pmatrix}
a_1 & a_2 & \cdots & a_n
\end{pmatrix} \in \mathbb{Z}^{d \times n}$ be such that ${\rm rank}(A) = d$ and $\mathbb{Z}A = \mathbb{Z}^d$. Let $P_A \subset \mathbb{R}^d$ be the polytope obtained as the convex hull of the lattice points $\{ \pm a_1, \ldots, \pm a_n \} \subset \mathbb{Z}^d$. For any $b \in \mathbb{C}^n$, the variety ${\cal L}_{A,b}$ is irreducible of dimension $d$. Moreover, for generic $b \in \mathbb{C}^n$, the degree of ${\cal L}_{A,b}$ is $\deg ({\cal L}_{A,b}) = 2^{-{\rm CL}_A} \, d! \, {\rm vol}(P_A)$, where ${\rm vol}(\cdot)$ denotes the euclidean volume and ${\rm CL}_A$ is the number of zero entries of a generic vector in the kernel of $A: \mathbb{R}^n \rightarrow \mathbb{R}^d$. 
\end{theorem}

The quantity $d!\,  {\rm vol}(P_A)$ is called the normalized volume of $P_A$. 
The polytope $P_A$ in Example \ref{ex:circleintro} is the line segment $[-1,1]$, with normalized volume two, which is the degree of the circle. The polygon $P_A$ in Example \ref{ex:elliptopeintro} is a hexagon with normalized volume $6 = \deg {\cal S}_A$. 
The notion of ``generic $b$'' in Theorem \ref{thm:mainintro} will be made more precise in Section \ref{sec:2}, and we give a formula for $\deg({\cal L}_{A,b})$ which holds for any $A,b$, but requires more notation (Theorem \ref{thm:degree}). We show that the hypersurface ${\cal C}_A$ obtained from the incidence matrix $A$ of the $n$-cycle graph $C_n$ is the cycle polynomial from \cite[Section 4.2]{sturmfels2010multivariate}. Using Theorem \ref{thm:degree} and a result from \cite{chen2018counting}, we prove a formula for the degree of the cycle polynomial (Proposition \ref{prop:degreecycle}). 

In Section \ref{sec:3}, we construct a matrix with polynomial entries whose maximal minors form a set of set-theoretic defining equations of ${\cal L}_{A,b}$ (Theorem \ref{thm:RankCondition}). In particular, a point $x^* \in \mathbb{C}^n$ belongs to ${\cal L}_{A,b}$ if and only if that matrix is not of full rank when evaluated at $x^*$. 

Section \ref{sec:4} establishes the role of Lissajous varieties of type ${\cal S}_A$ in computing steady state angles for the Kuramoto equations of coupled oscillators. In this case, $A$ is the incidence matrix of a graph $G$ which encodes the coupling. We will see that the equilibrium angles correspond to the intersection points of ${\cal S}_A$ with an affine-linear space of the form $A \, x = \omega$. In particular, the degree of ${\cal S}_A$ bounds the number of isolated solutions. 

In Section \ref{sec:optim}, we generalize the Kuramoto equations and construct dynamical systems whose steady state varieties are linear sections $A \, x = \omega$ of a Lissajous variety. Under certain assumptions on $\omega$, we show that one of the intersection points in ${\cal L}_{A,b} \cap \{A \,x  = \omega\}$ corresponds to a stable equilibrium. It is the unique solution to a convex optimization problem. For varying $\omega$, these equilibria parametrize a subset of the Lissajous variety,  called its positive part.

In Section \ref{sec:5}, we study the discriminant of the equilibrium equations ${\cal L}_{A,b} \cap \{A \,x  = \omega\}$ in the parameters $\omega$. We call this the Lissajous discriminant. We bound its degree and, when $A$ comes from a graph $G$, we describe its symmetries in terms of those of $G$. This is a first step in the bifurcation analysis of our dynamical system introduced in Section \ref{sec:optim}, see Example \ref{ex:bifur}.

\paragraph{Related work.} Replacing ``${\rm cos}$'' by ``${\rm exp}$'' in \eqref{eq:XAb}, we obtain the affine toric variety $Y_A$ of the matrix $A$ up to scaling the coordinates by ${\rm exp}(b_1), \ldots, {\rm exp}(b_n)$. Such scaled toric varieties appear in our study of the defining equations of ${\cal L}_{A,b}$, see Sections \ref{sec:3} and \ref{sec:4}.
The variety ${\cal C}_A = {\cal L}_{A,{\bf 0}}$ is called a \emph{Chebyshev variety} associated with $A$ in \cite[Section 5]{bel2024chebyshev}. This is motivated by the fact that for $d = 1$, the curve ${\cal C}_A$ is alternatively obtained from a parametrization by Chebyshev polynomials of the first kind. Previous work on the algebraic geometry of Kuramoto equations includes \cite{chen2018counting,harrington2023kuramoto,mehta2015algebraic}. These works use a different ``algebraic geometrization'' of the equations, see Equation (3) in \cite{mehta2015algebraic} and Definition 1.4 in \cite{harrington2023kuramoto}.
Some Lissajous discriminants were studied using machine learning techniques in \cite[Section 5.3]{bernal2023machine}.

\paragraph{Acknowledgements.} We are grateful to Monique Laurent for helpful comments on a previous version of this manuscript. We thank Georgios Korpas, Hal Schenck and Rainer Sinn for useful conversations. This work has been supported by the European Union’s HORIZON–MSCA-2023-DN-JD programme under
the Horizon Europe (HORIZON) Marie Skłodowska-Curie Actions, grant agreement 101120296 (TENORS).

\normalsize

\section{Dimension and degree} \label{sec:2}

As mentioned in the Introduction, if the matrices $A_1 \in \mathbb{Q}^{d_1 \times n}, A_2 \in \mathbb{Q}^{d_2 \times n}$ have the same row span, then we have an equality of Lissajous varieties ${\cal L}_{A_1,b} = {\cal L}_{A_2,b}$. In particular, after clearing denominators, we may assume that the matrix $A$ has integer entries. We fix $A \in \mathbb{Z}^{d \times n}, b \in \mathbb{C}^n$.

Choosing coordinates on ${\rm Row}(A)$, we can parametrize ${\cal L}_{A,b}$ as follows: 
\[ \phi_{A,b} \, : \, \mathbb{C}^d \longrightarrow \mathbb{C}^n, \quad t = (t_1, \ldots, t_d) \longmapsto ( {\rm cos}(a_1 \cdot t - b_1 \tfrac{\pi}{2}), \ldots, {\rm cos}(a_n \cdot t - b_n \tfrac{\pi}{2} )). \]
Here $a_j \in \mathbb{Q}^d$ is the $j$-th column of $A$ and $a_j \cdot t$ is the standard dot product. We clearly have ${\cal L}_{A,b} = {\rm im} \, \phi_{A,b}$. We obtain a rational parametrization of our Lissajous variety as follows.  Set 
\[ \beta_\ell \, = \,  e^{-i b_\ell \tfrac{\pi}{2}}, \,  \, \ell = 1, \ldots, n \quad \text{with} \quad  i = \sqrt{-1}.\]  
Using Euler's identity $\cos( \theta )= \tfrac{1}{2}(e^{i\theta} + e^{-i \theta})$, we see $\phi_{A,b}(t) = \psi_{A,b}(e^{i t_1}, \ldots, e^{i t_d})$, with 
\begin{equation} \label{eq:psiAb} \psi_{A,b} \, : \, (\mathbb{C}^*)^d \rightarrow \mathbb{C}^n, \quad v = (v_1, \ldots, v_d) \longmapsto \Big ( \frac{\beta_1 v^{a_1} + \beta_1^{-1}v^{-a_1} }{2}, \ldots, \frac{\beta_n v^{a_n} + \beta_n^{-1}v^{-a_n} }{2} \Big ).\end{equation}
Here $v^{a_j} = \prod_{k = 1}^d v_k^{a_{kj}}$. We show below that ${\rm im} \,  \phi_{A,b} = {\rm im} \, \psi_{A,b}$, and therefore ${\cal L}_{A,b} = {\rm im} \, \psi_{A,b}$. 

\begin{lemma} \label{lem:noclosure}
    The set ${\cal  L}_{A,b} = {\rm im} \, \psi_{A,b} = {\rm im} \,  \phi_{A,b} \subseteq \mathbb{C}^n$ is a closed affine subvariety of $\mathbb{C}^n$. Moreover, ${\cal L}_{A,b}$ is irreducible of dimension ${\rm rank}(A)$. 
\end{lemma}
\begin{proof}
    The equality ${\rm im} \, \phi_{A,b} = {\rm im} \, \psi_{A,b}$ follows from the fact that $\phi_{A,b}: \mathbb{C}^d \rightarrow \mathbb{C}^n$ is the composition of $\psi_{A,b}: (\mathbb{C}^*)^d \rightarrow \mathbb{C}^n$ with the surjective map $\mathbb{C}^d \rightarrow (\mathbb{C}^*)^d$ given by $t \mapsto (e^{it_1}, \ldots, e^{it_d})$. The image ${\rm im} \, \phi_{A,b} = {\rm im} \, \psi_{A,b}$ only depends on the row span of $A$. After applying integer row operations and column permutations to $A$, and after dropping rows whose entries are all zero, we obtain a matrix $\tilde{A}$ satisfying ${\rm im} \,  \psi_{A,b} = {\rm im} \, \psi_{\tilde{A},b}$ of the form
    \[ \tilde{A} \, = \, \begin{pmatrix}
        a_{11} & 0 & \cdots & 0 & a_{1,r+1} & \cdots & a_{1,n} \\ 
        0 & a_{22} & \cdots & 0 & a_{2,r+1} & \cdots & a_{2,n} \\ 
        \vdots & \vdots & \ddots & \vdots & \vdots & \cdots & \vdots \\ 
        0 & 0 & \cdots & a_{rr} & a_{r,r+1} & \cdots & a_{r,n}
    \end{pmatrix} \quad \in \, \mathbb{Z}^{r \times n},\]
    for some positive integers $a_{jj}, j = 1, \ldots, r$. Here $r \leq d$ is the rank of $A$. We want to show that the image of $\psi_{ \tilde A,b} \colon (\mathbb{C}^*)^r \rightarrow \mathbb{C}^n$ is Zariski closed.
    Since $\mathbb{C}$ is algebraically closed, the Zariski closure of ${\rm im} \, \psi_{ \tilde A,b}$ equals its Euclidean closure. Hence, if $x^* \in \mathbb{C}^n$ lies in the closure of ${\rm im} \, \psi_{\tilde A,b}$, there is a smooth path $\{ x^*(t) : t \in (0,1] \}$ contained in ${\rm im} \, \psi_{A,b}$ whose limit for $t \to 0$ is $x^*$. This lifts to a smooth path $v(t) = (v_1(t), \ldots, v_r(t))$ in $ (\mathbb{C}^*)^r$ with $\psi_{\tilde A,b}(v(t)) = x^*(t)$. We~have
    \begin{equation} \label{eq:limit}  x^*_j \, = \,  \lim_{t \rightarrow 0} \tfrac{1}{2} (\beta_j \, v_j(t)^{a_{jj}} + \beta_j^{-1} \, v_j(t)^{-a_{jj}}) \quad  \text{ for } j = 1, \ldots, r .\end{equation}
    Since the limit \eqref{eq:limit} lies in $\mathbb{C}$, we have $\lim_{t \rightarrow 0} v_j(t) \in \mathbb{C}^*$. Indeed, if the limit $\lim_{t \rightarrow 0} v_j(t)$ does not exist in $\mathbb{C}^*$, then $v_j(t)$ approaches $0$ or $\infty$ for $t \rightarrow 0$. In both cases, $x_j(t)$ would not converge in $\mathbb{C}$. Hence, we have that $x^* = \lim_{t \rightarrow 0} \psi_{\tilde A,b}(v(t)) =\psi_{\tilde A,b}(\lim_{t \rightarrow 0}v(t)) \in {\rm im} \, \psi_{\tilde A,b}$. 

    We have now shown that ${\cal L}_{A,b} = {\rm im} \, \psi_{A,b} = {\rm im} \, \psi_{\tilde A, b}$ is closed in $\mathbb{C}^n$. Since ${\cal L}_{A,b}$ is unirational, it is irreducible. It is also clear that the projection of ${\rm im} \, \psi_{\tilde A, b}$ to the first $r$ coordinates is surjective onto $\mathbb{C}^r$. Hence, we have $\dim {\cal L}_{A,b} = r = {\rm rank}(A)$.
\end{proof}

The equality ${\cal L}_{A,b} = {\rm im}\,  \psi_{A,b}$ implies an algorithm for computing the defining ideal of~${\cal L}_{A,b}$:

\begin{enumerate}
\itemsep0em 
    \item Work in the ring $\mathbb{C}[v_1,\ldots, v_d, w_1, \ldots, w_d, x_1,\ldots,x_n]$ with $2d+n$ variables.
    \item For each column $a_i$ of $A$, write $a_i = a_i^+ - a_i^-$ with $a_i^{+}, a_i^- \in \mathbb{N}^d$.
    \item Define $I = \langle 2 \, x_j - \beta_1 v^{a_j^+}w^{a_j^-} - \beta_1^{-1}v^{a_j^-} w^{a_j^+}, \, \, v_kw_k-1, \, \, j = 1, \ldots, n, \, \, k = 1, \ldots, d\rangle$.
    \item Compute the elimination ideal $I \cap \mathbb{C}[x_1, \ldots, x_n]$. The result equals $I({\cal L}_{A,b})$.
\end{enumerate}
The ideal $I$ in step 3 has $n + d$ generators. Adding variables $w_k$ and imposing $v_kw_k-1$ is an effective way of working in the Laurent polynomial ring $\mathbb{C}[v_1^{\pm 1}, \ldots, v_d^{\pm 1}]$: $w_k$ plays the role of $v_k^{-1}$. The correctness of this algorithm follows from the parametrization $\psi_{A,b}$. Notice that, for the ground field we can use any field extension of $\mathbb{Q}$ containing $\beta_1, \ldots, \beta_n$. In particular, if $\beta_\ell$ is rational for $\ell = 1, \ldots, n$, then one can replace $\mathbb{C}$ by $\mathbb{Q}$ in step 1. This happens when $b = \mathbf{0}$ and ${\cal L}_{A,b} = {\cal C}_A$. If $b = \mathbf{1}$ and ${\cal L}_{A,b} = {\cal S}_A$, then one can work over the field $\mathbb{Q}[i] = \mathbb{Q}[z]/\langle z^2 + 1 \rangle$. 

The integer matrix $A$ defines an affine toric variety $Y_A \subseteq \mathbb{C}^n$ obtained as the closure of the image of the Laurent monomial map $v \mapsto (v^{a_1}, \ldots, v^{a_n})$, $v \in (\mathbb{C}^*)^d$ \cite{CoxLittleSchenck2011,Telen2025MonomialMaps}. Its ideal is 
\[ I_A \, = \, \langle y^{u} - y^{w} \, : \, u, w \in \mathbb{N}^n \, \, \text{and} \, \,  A(u-w) = 0 \rangle \, \subseteq \, \mathbb{C}[y_1, \ldots, y_n].\]
A modified or \emph{scaled} toric variety $Y_{A,\beta}$ is obtained from $A \in \mathbb{Z}^{d \times n}$, $\beta \in (\mathbb{C}^*)^n$ as follows:
\[ Y_{A,\beta} \, = \, \{ (\beta_1 \, y_1, \ldots, \beta_n \, y_n) \, : \,  y \in Y_A \} \, \subseteq \, \mathbb{C}^n. \]
This variety is parametized by $v \mapsto (\beta_1 v^{a_1}, \ldots, \beta_n v^{a_n})$. One checks that if $I_A = I(Y_{A}) = \langle y^{u_1} - y^{w_1}, \ldots, y^{u_r}-y^{w_r} \rangle$, then $I_{A,\beta} = I(Y_{A,\beta})$ is generated by $\beta^{w_k} y^{u_k} - \beta^{u_k}y^{w_k}, k = 1, \ldots, r$. 

\begin{theorem}\label{thm:dimension}
    The Lissajous variety ${\cal L}_{A,b} \subseteq \mathbb{C}^n$ is the image of the variety
    \[ {\cal Y}_{A,b} \, = \, \{(x,y) \in \mathbb{C}^n \times (\mathbb{C}^*)^n \, : \, y \in Y_{A,\beta} \text{ and } x_j \, = \, \tfrac{1}{2}(y_j + y_j^{-1}) \text{ for } j = 1, \ldots,n \}, \]
    under the coordinate projection $\pi_{A,b}: {\cal Y}_{A,b}  \rightarrow \mathbb{C}^n$. 
    Here $\beta = (e^{-i b_1 \tfrac{\pi}{2}}, \ldots, e^{-i b_n \tfrac{\pi}{2}}) \in (\mathbb{C}^*)^n$.
\end{theorem}
\begin{proof}
The proof is an adaptation of the proof of \cite[Theorem 5.2]{bel2024chebyshev}. 
The map $y \mapsto (\tfrac{1}{2}(y_1 + y_1^{-1}), \ldots, \tfrac{1}{2}(y_n + y_n^{-1}), y)$ induces an isomorphism $Y_{A,\beta} \cap (\mathbb{C}^*)^n \simeq {\cal Y}_{A,b}$.
Hence, we know from basic toric geometry that ${\cal Y}_{A,b}$ is a torus of dimension ${\rm rank}(A)$ \cite{CoxLittleSchenck2011,Telen2025MonomialMaps}. Composing this map with the coordinate projection $\pi_{A,b}$ and setting $y_j = \beta_j v^{a_j}$, we obtain precisely the map $\psi_{A,b}$ with image ${\cal L}_{A,b}$. This implies that ${\cal L}_{A,b} = \pi_{A,b}({\cal Y}_{A,b})$. 
\end{proof}
Deleting spurious rows if necessary, we may assume that $A$ has rank $d$. To state a degree formula for ${\cal L}_{A,b}$, let $[\mathbb{Z}^d : \mathbb{Z}A]$ be the index of the lattice $\mathbb{Z}A$ generated by the columns of $A$ in the ambient lattice $\mathbb{Z}^d$. The \emph{degree} $\deg \phi$ of a dominant morphism $\phi: X \rightarrow Y$ between irreducible varieties of dimension $d$ is the cardinality of a generic fiber. 

\begin{theorem} \label{thm:degree}
Let $A \in \mathbb{Z}^{d \times n}$ be such that ${\rm rank}(A) = d$. Let $P_A \subset \mathbb{R}^d$ be the polytope obtained as the convex hull of the lattice points $\{ \pm a_1, \ldots, \pm a_n \} \subset \mathbb{Z}^d$. For any $b \in \mathbb{C}^n$, we have 
\[ \deg ({\cal L}_{A,b}) \, = \,  \frac{ d! \, {\rm vol}(P_A)}{\deg \pi_{A,b} \cdot [\mathbb{Z}^d : \mathbb{Z} A]} \, ,\]
where ${\rm vol}(\cdot)$ denotes the euclidean volume and $\pi_{A,b}: {\cal Y}_{A, b} \rightarrow {\cal L}_{A,b}$ is as in Theorem \ref{thm:dimension}. 
\end{theorem}
\begin{proof}
    This is a generalization of \cite[Theorem 5.3]{bel2024chebyshev}. The degree of ${\cal L}_{A,b}$ is the cardinality of 
    \begin{equation*} \label{ex:xsolutions}
        S_x \, = \, \{ x \in {\cal L}_{A,b} \, : \, c_{j0} + c_{j1} \, x_1 + \cdots + c_{jn} \, x_n \,  =  \, 0, \, \, \,  j = 1, \ldots, d \}
    \end{equation*}
    for generic complex coefficients $c_{jk}$. We compare the set $S_x$ to the set $S_v$ given by 
    \[ S_v \, = \, \{ v \in (\mathbb{C}^*)^d \, : \, c_{j0} + \tfrac{c_{j1}}{2}(\beta_1 v^{a_1} + \beta_1^{-1} v^{-a_1}) + \cdots + \tfrac{c_{jn}}{2}(\beta_n v^{a_n} +  \beta_n^{-1} v^{-a_n}) \, = \, 0, \, \, \, j = 1, \ldots, d \}. \]
    By Lemma \ref{lem:noclosure}, we have $\psi_{A,b}(S_v) = S_x$ for any choice of $c_{jk}$. 
    Moreover, this correspondence is generically $\deg \psi_{A,b}$-to-one. 
    The denominator $\deg \pi_{A,b} \cdot [\mathbb{Z}^d : \mathbb{Z} A]$ in our degree formula is the degree of $\psi_{A,b}$, since $\psi_{A,b}$ is the composition of the map $v \mapsto (\psi_{A,b}(v), \beta_1 v^{a_1}, \ldots , \beta_n v^{a_n})$ 
    with $\pi_{A,b}$ (Theorem \ref{thm:dimension}). 
    Indeed, the first map has the same degree as $v \mapsto ( \beta_1 v^{a_1}, \ldots, \beta_n v^{a_n})$, which parametrizes $Y_{A,\beta}$ and whose degree is $[\mathbb{Z}^d: \mathbb{Z} A]$, see \cite{Telen2025MonomialMaps}. 
    
    It remains to show that $S_v$ consists of $d! \, {\rm vol}(P_A)$ points. This is an application of Kushnirenko's theorem \cite[Th\'eor\`eme III']{kouchnirenko1976polyedres}, which predicts the maximal (and expected) number of solutions to a system of Laurent polynomial equations with identical monomial support. In order to apply this theorem, we must ensure that our equations are \emph{non-degenerate} with respect to the polytope $P_A$. This means that none of the facial subsystems have a solution in $(\mathbb{C}^*)^d$. By our assumption that ${\rm rank}(A) = d$, the centrally symmetric polytope $P_A$ is full-dimensional and it contains the origin in its interior. This implies that no  face of $P_A$, except $P_A$ itself, contains both the lattice points $a_j$ and $-a_j$. Therefore, there are no dependencies among the coefficients appearing in any facial subsystem. Hence, for generic $c_{jk}$, our equations are indeed non-degenerate and Kushnirenko's upper bound $|S_v| \leq d! \, {\rm vol}(P_A)$ is attained.
\end{proof}

To complete the proof of Theorem \ref{thm:mainintro}, we must introduce some more notation. We write ${\rm Circ}(A)$ for the set of \emph{circuits} of $A$. That is, ${\rm Circ}(A)$ is the set of all subsets of $[n] = \{1, \ldots, n\}$ indexing a minimal set of linearly dependent columns of $A$. For any integer vector $m = (m_1, \ldots, m_n) \in \mathbb{Z}^n$, we define the \emph{support} of $m$ as follows: ${\rm supp}(m) = \{ j \in [n] \, : \, m_j \neq 0 \}$. For each circuit $C \in {\rm Circ}(A)$ there is a unique (up to sign) integer vector $m^C = (m^C_1, \ldots, m^C_n) \in \mathbb{Z}^n$ of minimal length satisfying $A \cdot m_C = 0$ and ${\rm supp}(m^C) = C$. This vector encodes the unique linear relation between the columns indexed by the circuit $C$. A \emph{coloop} of the matrix $A$ is an element $j \in [n]$ which does not belong to any circuit. We denote the set of coloops by ${\rm Coloops}(A) \subseteq [n]$, and its cardinality by ${\rm CL}_A = \#{\rm Coloops}(A)$. The integer ${\rm CL}_A$ is the number of zero entries of a generic vector in the kernel of $A$, as in Theorem \ref{thm:mainintro}.

\begin{theorem} \label{thm:genericbeta}
    Assume that $\mathbb{Z} A = \mathbb{Z}^d$, so that ${\rm rank}(A) = d$ and $[\mathbb{Z}^d: \mathbb{Z} A] = 1$. If for each circuit $C \in {\rm Circ}(A)$ the vector $b = (b_1, \ldots, b_n) \in \mathbb{C}^n$ satisfies 
\begin{equation} \label{eq:noteven}
    m^C \cdot b \, = \, \sum_{j \in C} m^C_j \, b_j \quad \text{is not an even integer,}
\end{equation}
then $\deg \pi_{A,b} = 2^{{\rm CL}_A}$ and the formula from Theorem \ref{thm:degree} simplifies to $\deg({\cal L}_{A,b}) = \frac{d! \, {\rm vol}(P_A)}{2^{{\rm CL}_A}}$. 
\end{theorem}

\begin{proof}
    To investigate the degree of $\pi_{A,b}$, pick a generic point $(x,y) \in {\cal Y}_{A,b} \subseteq \mathbb{C}^n \times (\mathbb{C}^*)^n$ and notice that the only candidates for the points in the fiber $\pi_{A,b}^{-1}(\pi_{A,b}(x,y))$ are the $2^n$ points $(x,y_1^{\pm 1}, \ldots, y_n^{\pm 1})$. For any subset $J \subseteq [n]$, let $\tilde{y} \in (\mathbb{C}^*)^n$ be given by 
    \[ \tilde{y}_j \, = \, \begin{cases}
        y_j & j \notin J \\ 
        y_j^{-1} & j \in J
    \end{cases}.\]
    To test whether $(x,\tilde{y}) \in \pi_{A,b}^{-1}(\pi_{A,b}(x,y))$, we need to check whether $\tilde{y} \in Y_{A,\beta} \cap (\mathbb{C}^*)^n$. 
    
    The variety $Y_{A,\beta} \cap (\mathbb{C}^*)^n$ is given by one binomial equation for each circuit of $A$: 
    \[ Y_{A,\beta} \cap (\mathbb{C}^*)^n \, = \, \{ y \in (\mathbb{C}^*)^n \, : \, y^{m^C} = \beta^{m^C} \, \, \, \text{for all} \, \, \, C \in {\rm Circ}(A)\}.  \]
    If $J \subseteq {\rm Coloops}(A)$, then the coordinates $y_j, j \in J$ do not appear in our binomial equations and $\tilde{y}$ lies on $Y_{A,\beta} \cap (\mathbb{C}^*)^n$. If $J \not \subseteq {\rm Coloops}(A)$ (in particular, $J \neq \emptyset$), pick a circuit $C$ such that $J \cap C \neq \emptyset$. If both $y$ and $\tilde{y}$ lie on $Y_{A,\beta}$, then 
    \begin{equation} \label{eq:contrad} y^{m^C} \cdot \tilde{y}^{m^C} \, = \, \prod_{j \in C \setminus J} y_j^{2 \, m_j^C} \, = \, \beta^{\, 2 \, m^C}.\end{equation}
    The second equality in \eqref{eq:contrad} holds for generic $y \in Y_{A,\beta}$ if and only if, for all $v \in (\mathbb{C}^*)^d$, we have
    \[ \prod_{j \in C \setminus J} (\beta_j v^{a_j})^{2 \, m^C_j} \, = \, \beta^{\, 2 \, m^C}, \quad \text{which implies} \quad v^{\, 2  \sum_{j\in C \setminus J}m^C_j \, a_j} \, = \, \prod_{j \in J \cap C} \beta_j^{\, 2 \, m^C_j}.\]
    In particular, we must have $\sum_{j \in C\setminus J} m_j^C \, a_j = 0$. If $J \cap C \subsetneq C$, then this contradicts the fact that $C$ is a circuit.
    If $J \cap C = C$, then the equality fails if $\beta^{\, 2 \, m^C} \neq 1$. Taking the logarithm on both sides of this inequality, we obtain precisely the condition \eqref{eq:noteven}. We have shown that, under the condition \eqref{eq:noteven} for every circuit $C \in {\rm Circ}(A)$, the points in $\pi_{A,b}^{-1}(\pi_{A,b}(x,y))$ are the points of the form $\tilde{y}$ corresponding to $J \subseteq {\rm Coloops}(A)$. Hence, the cardinality is $2^{{\rm CL}_A}$. 
\end{proof}

\begin{remark}
    When $b = \mathbf{0}$, the condition \eqref{eq:noteven} is never satisfied. In fact, in this case, we have $\deg \pi_{A,b} \geq 2$, as the fiber $\pi_{A,\mathbf{0}}^{-1}(\pi_{A,\mathbf{0}}(x,y))$ contains $(x,y)$ and $(x,y^{-1})$ \cite[Remark 5.4]{bel2024chebyshev}. The second point is obtained as $\tilde{y}$ for $J = [n]$ in the notation of the proof of Theorem \ref{thm:genericbeta}.
\end{remark}

We end the section by applying our degree formula to a family of Chebyshev hypersurfaces which arises when studying elliptopes of cycle graphs \cite{SolusUhlerYoshida2016,sturmfels2010multivariate}. This is relevant in semidefinite completion problems, see for instance \cite{Laurent1997}. Rephrasing \cite[Equation (28)]{sturmfels2010multivariate} in our notation, the $n$-th \emph{cycle polynomial} $\Gamma_n'$ is the defining equation of the Chebyshev hypersurface ${\cal C}_{A_n}$,~with 
\begin{equation} \label{eq:cycleAn} A_n \, = \, \begin{pmatrix}
    1 & 0 & \cdots & 0 & -1 \\ 
    0 & 1 & \cdots & 0 & -1 \\ 
    \vdots & \vdots & \ddots & \vdots & \vdots \\ 
    0 & 0 & \cdots & 1 & -1
\end{pmatrix} \quad \in \, \, \mathbb{Z}^{(n-1) \times n}.\end{equation}
Notice that $\Gamma_n'$ is defined up to scale. Its degree was computed up to $n = 11$ in \cite[Table 1]{sturmfels2010multivariate}. In the paragraph preceding Conjecture 4.9 in \cite{sturmfels2010multivariate}, it is stated that a closed formula for $\deg(\Gamma_n')$ is not known. The next proposition provides such a formula. 

\begin{proposition} \label{prop:degreecycle}
    The degree of the $n$-th cycle polynomial $\Gamma_n'$ is given by 
    \begin{equation} \label{eq:volumecycle} \deg (\Gamma_n') \, = \,
        \frac{n}{2} \binom{n-1}{\lfloor \tfrac{1}{2}(n-1) \rfloor} . \end{equation}
\end{proposition}
\begin{proof}
    By Theorem \ref{thm:degree}, we have $\deg(\Gamma_n') = (\deg \pi_{A_n,{\bf 0}})^{-1} (n-1)! \, {\rm vol}(P_{A_n})$. Let $y \in Y_{A_n} \cap (\mathbb{C}^*)^n = \{ y \in (\mathbb{C}^*)^n \, : \, y_1\cdots y_{d+1} = 1 \}$ be a generic point on the affine toric hypersurface associated to $A_n$. The degree of $\pi_{A,{\bf 0}}$ is the number of points among $\{(y_1^{\pm 1}, \ldots, y_n^{\pm 1}) \}$ lying on $Y_{A_n}$. That number is two: only $(y_1, \ldots, y_n)$ and $(y_1^{-1}, \ldots, y_n^{-1})$ lie on $Y_{A_n}$. The number $(n-1)! {\rm vol}(P_{A_n})$ was computed in \cite[Theorem 14]{chen2018counting}, and multiplying with $1/2$ gives \eqref{eq:volumecycle}.
\end{proof}

\section{Determinantal equations} \label{sec:3}

Our goal in this section is to derive defining equations for ${\cal L}_{A,b}$ from rank conditions on a matrix with entries in $\mathbb{C}[x]$. Motivated by the equations $2\, x_j = y_j + y_j^{-1}$ appearing in Theorem \ref{thm:dimension}, we work in the following setup. Let $K = \overline{\mathbb{C}(x)}$ be the algebraic closure of the field of rational functions in $x = (x_1, \ldots, x_n)$, and consider the following ideal: 
\[ J \, = \, \langle 2 \, x_1 - y_1 - y_1^{-1}, \, \ldots, \, 2\, x_n - y_n - y_n^{-1}\rangle \, \, \subset \, \, K[y^{\pm}] = K[y_1^{\pm 1}, \ldots, y_n^{\pm 1}]. \]
The affine variety $V(J)$ defined by $J$ in $(K^*)^n$ is zero-dimensional and consists of $2^n$ distinct points, each with multiplicity one. Hence, $J$ is radical and the quotient ${\cal A} = K[y^{\pm 1}]/J$ is a $K$-vector space of dimension $2^n$. Let $f_j = 2\, x_j - y_j - y_j^{-1}$ be the $j$-th generator of $J$. We have 
\begin{equation} \label{eq:tensorproduct}
{\cal A} \, = \, \frac{K[y^{\pm 1}]}{\langle f_1,\dotsc,f_n\rangle}\, \simeq \, \frac{K[y_1^{\pm 1}]}{\langle f_1\rangle}\otimes_K\, \cdots \, \otimes_K\frac{K[y_n^{\pm 1}]}{\langle f_n\rangle} \, . 
\end{equation}
Below we write $[h]$ for the residue class of $h \in K[y^{\pm 1}]$ in ${\cal A}$. Multiplication with an element $g \in K[y^{\pm 1}]$ gives a $K$-linear endomorphism $M_g : {\cal A} \rightarrow {\cal A}$, where $M_g([h]) = [gh]$. Once we fix a $K$-basis for ${\cal A}$, such a map is represented by a $2^n \times 2^n$ matrix with entries in $K$. We shall now describe such matrices for a basis compatible with the tensor product structure \eqref{eq:tensorproduct}.

For a moment, set $n = 1$. Let us fix the $K$-basis $\{[1],[ y_1] \}$ for the algebra ${\cal A} = K[y_1^{\pm 1}]/\langle f_1 \rangle$. We claim that, in this basis, multiplication with $y_1$ is given by the $2 \times 2$ matrix    \[
    M_{y_1} \, = \, \begin{pmatrix}
        0 & -1\\
        1 & 2 \, x_1
    \end{pmatrix}.
    \]
Indeed, the first column reads $[y_1 \cdot 1] = \textcolor{blue}{0} \cdot [1] + \textcolor{blue}{1} \cdot [y_1]$, and the second column reads 
\[ [y_1 \cdot y_1] \, = \, [y_1^2] \, = \, \textcolor{blue}{-1} \cdot [1] + \textcolor{blue}{2 \, x_1} \cdot [y_1] ,\]
where the second equality follows from $[y_1f_1] = 0$. 
Multiplication with a general element $g \in K[y_1^{\pm 1}]$ is given by $M_g = g(M_{y_1})$, where $g(M_{y_1})$ denotes the matrix obtained by substituting $M_{y_1}$ for $y_1$ in the monomial expansion of $g$.

By \eqref{eq:tensorproduct}, a general element $a \in {\cal A}$ is given by a finite sum $a = \sum_k a_1^{(k)}\otimes \dotsc \otimes a_n^{(k)}$, where 
\[ a_j^{(k)}\in {\cal A}_j = K[y_j^{\pm 1}]/\langle f_j\rangle. \]
Multiplication by a single variable, say $y_1$, satisfies
\[
M_{y_1}(a) \, =\,  M_{y_1} \Big (\sum_k a_1^{(k)}\otimes \dotsc\otimes a_n^{(k)} \Big) \, =\, \sum_k M_{y_1}( a_1^{(k)})\otimes a_2^{(k)}\otimes \cdots \otimes a_n^{(k)}.
\]
Hence, fixing the basis $\{ [1],[y_j]\}$ for ${\cal A}_j$ and the corresponding tensor product basis for ${\cal A} = {\cal A}_1 \otimes_K \cdots \otimes_K {\cal A}_n$, our previous observation gives
\[
    M_{y_1} = 
    \begin{pmatrix}
        0 & -1\\
        1 & 2 \, x_1
    \end{pmatrix}
    \otimes {\rm id}_2 \otimes \dotsc \otimes {\rm id}_2, 
\]
where ${\rm id}_2$ is the $2 \times 2$ identity matrix and $\otimes$ is the Kronecker product of matrices.
The matrices $M_{y_2},\dotsc,M_{y_n}$ are found in an analogous way. 
By definition, our matrices $M_{y_1}, \ldots, M_{y_n}$ are pairwise commuting. 
Hence, for any Laurent polynomial $g \in K[y^{\pm 1}]$, the matrix $g(M_{y_1}, \ldots, M_{y_n})$ obtained by substituting $M_{y_j}$ for $y_j$ in the monomial expansion of $g$ is well-defined and represents the map $M_g$.
Since the entries of both $M_{y_j}$ and $M_{y_j}^{-1}$ are polynomials in $x_j$, we have that $M_g \in \mathbb{C}[x]^{2^n \times 2^n}$ for any $g \in \mathbb{C}[y^{\pm 1}] \subset K[y^{\pm 1}]$. 
Here is the main theorem of this section.
\begin{theorem}\label{thm:RankCondition}
    Fix $A\in \mathbb{Z}^{d \times n}$ and $b \in \mathbb{C}^n$. Let $\beta\in(\mathbb{C}^*)^n$ as usual, and $Y_{A,\beta} \subseteq \mathbb{C}^n$ be the corresponding scaled affine toric variety.
    If $g_1,\dotsc,g_r$ generate the ideal of $Y_{A,\beta} \cap (\mathbb{C}^*)^n$ in $\mathbb{C}[y_1^{\pm 1}, \ldots, y_n^{\pm 1}]$, then 
    \begin{equation} \label{eq:rankrep}
    {\cal L}_{A,b} \, = \, \{x\in\C^n : {\rm rank}(M_{g_1}|\dotsc|M_{g_r})<2^n\},
    \end{equation}
    where $M_{g_k}$ is a $2^n \times 2^n$-matrix with entries in $\mathbb{C}[x]$ representing multiplication by $g_k$ in ${\cal A}$.
\end{theorem}
\begin{remark}
    It is easier to obtain generators for the ideal of $Y_{A,\beta} \cap (\mathbb{C}^*)^n$ than for the toric ideal of $Y_{A,\beta}$. First, one computes a $\mathbb{Z}$-kernel of $A$. This can be found, for instance, via the Smith normal form of $A$. If this kernel is generated by $m_1, \ldots, m_r \in \mathbb{Z}^n$, with $r = n - {\rm rank}(A)$, then the binomials $g_1, \ldots, g_r$ in Theorem \ref{thm:RankCondition} are given by $y^{m_k} - \beta^{m_k}, k = 1, \ldots, r$ \cite{Telen2025MonomialMaps}. To avoid computing matrix inverses, one may clear denominators by writing $m_k = u_k - w_k$ with $u_k, w_k \in \mathbb{N}^n$ and use $g_k = \beta^{w_k} y^{u_k} - \beta^{u_k} y^{w_k}$ instead.
\end{remark}

Theorem \ref{thm:RankCondition} implies that a set of defining equations for ${\cal L}_{A,b}$ is given by the maximal minors of the $2^n \times (r \, 2^n)$-matrix $(M_{g_1}|\dotsc|M_{g_r})$ obtained by concatenating the matrices $M_{g_j}$ as block columns.
In particular, if $A$ has rank $n-1$, then $Y_{A, b}$ is a hypersurface and its ideal $I_{A,b} = \langle g \rangle$ is principal. Hence, Theorem \ref{thm:RankCondition} provides a determinantal representation for ${\cal L}_{A,b}$.
Before proving Theorem \ref{thm:RankCondition}, we illustrate the statement in our running example.

\begin{example}
    Let $A$ be as in Example \ref{ex:elliptopeintro}. 
    Following the discussion above, we construct matrices representing multiplication with the variables in $K[y_1^{\pm 1}, y_2^{\pm 1}, y_3^{\pm 1}]/\langle f_1,f_2,f_3\rangle$:
    \[
    M_{y_1} = \begin{pmatrix}
        0 & -1\\
        1 & 2 \, x_1
    \end{pmatrix}
    \otimes {\rm id}_2 \otimes {\rm id}_2,\ \
    M_{y_2} = {\rm id}_2\otimes
    \begin{pmatrix}
        0 & -1\\
        1 & 2 \, x_2
    \end{pmatrix}
    \otimes {\rm id}_2,\ \text{ and }\ 
    M_{y_3} = {\rm id}_2 \otimes {\rm id}_2 \otimes
    \begin{pmatrix}
        0 & -1\\
        1 & 2 \, x_3
    \end{pmatrix}.
    \]
    The toric ideal of $Y_{A,\beta}$ is generated by $g= y_1 y_2y_3- \beta_1 \beta_2 \beta_3$. Evaluating this at our multiplication operators, we find
    $M_g = M_{y_1} M_{y_2} M_{y_3} -  \beta_1 \beta_2 \beta_3 \, {\rm id}_{8 \times 8}$, which gives the following~result: 
\[
\setlength\arraycolsep{1.8pt}
M_g \, = \, \begin{pmatrix}-\beta_1 \beta_2 \beta_3&0&0&0&0&0&0&-1\\0&-\beta_1 \beta_2 \beta_3&0&0&0&0&1&2 x_3\\0&0&-\beta_1 \beta_2 \beta_3&0&0&1&0&2 x_2\\0&0&0&-\beta_1 \beta_2 \beta_3&-1&-2 x_3&-2 x_2&-4 x_2 x_3\\0&0&0&1&-\beta_1 \beta_2 \beta_3&0&0&2 x_1\\0&0&-1&-2 x_3&0&-\beta_1 \beta_2 \beta_3&-2 x_1&-4 x_1 x_3\\0&-1&0&-2 x_2&0&-2 x_1&-\beta_1 \beta_2 \beta_3&-4 x_1 x_2\\1&2 x_3&2 x_2&4 x_2 x_3&2 x_1&4 x_1 x_3&4 x_1 x_2&8 x_1 x_2 x_3-\beta_1 \beta_2 \beta_3\end{pmatrix}\]
Setting $b = \mathbf{0}$ (i.e., $\beta = (1,1,1)$), the determinant evaluates to
\[
(\det M_g)_{|\beta = (1,1,1)} \, = \, 16 \, (x_1^2 - 2x_1x_2x_3 + x_2^2 + x_3^2 - 1)^2.
\]
This is the square of the defining equation of Cayley's cubic surface $\mathcal{C}_A = {\cal L}_{A,\mathbf{0}}$.
This computation shows that the equality \eqref{eq:rankrep} holds only set-theoretically; the maximal minors of $(M_{g_1} | \cdots | M_{g_r})$ do not necessarily generate a radical ideal. 
The equation \eqref{eq:SA} for $\mathcal{S}_A= {\cal L}_{A,\mathbf{1}}$ is obtained by evaluating $\det M_g$ at $b = \mathbf{1}$ or $\beta = (-i,-i,-i)$. In this case, and for generic $\beta$, $\det M_g$ is an irreducible polynomial of degree six. The case $\beta = (1,1,1)$ is exceptional. 
\end{example}

To prove Theorem \ref{thm:RankCondition}, we first state a classical theorem, which describes the eigenvalues and left eigenvectors of a multiplication map. We refer to \cite[Théorème 4.23]{MourrainElkadi2007ResolutionSystemesPolynomiaux} for details.
\begin{theorem}[Eigenvalue, eigenvector theorem]\label{thm:EigenvalueEigenvector}
    Let $\cal A$ be the coordinate ring of a zero-dimensional scheme with support $V=\{z_1,\dotsc,z_{\delta}\}$. 
    The eigenvalues of the multiplication map $M_g \colon \cal A\rightarrow \cal A$ are $g(z_1),\dotsc,g(z_{\delta})$, and the multiplicity of the eigenvalue $g(z_j)$ equals the multiplicity of $z_j$. Moreover, for each $j = 1, \ldots, \delta$, the evaluation map ${\rm ev}_{z_j}: {\cal A} \rightarrow \mathbb{C}$ given by ${\rm ev}_{z_j}([h]) = h(z_j)$ is a left eigenvector corresponding to the eigenvalue $g(z_j)$.
\end{theorem}
Once a basis for the vector space ${\cal A}$ is fixed, the evaluation map ${\rm ev}_{z_j}: {\cal A} \rightarrow \mathbb{C}$ is represented by a row vector of length $\dim {\cal A}$. The eigenvalue relation is ${\rm ev}_{z_j} \, M_g \, = \, g(z_j) \, {\rm ev}_{z_j}$. 
\begin{proof}[Proof of Theorem \ref{thm:RankCondition}]
    Let $W$ be the righthand side in \eqref{eq:rankrep}. For $x^* \in \mathbb{C}^n$, let us write $J_{x^*} \subseteq \mathbb{C}[y^{\pm 1}]$ for the ideal generated by $(f_1)_{x = x^*}, \ldots, (f_n)_{|x = x^*}$. The variety $V(J_{x^*}) \subset (\mathbb{C}^*)^n$ consists of at most $2^n$ points. Recall from Theorem \ref{thm:dimension} that ${\cal L}_{A,b} = \pi_{A,b}({\cal Y}_{A,b})$ and 
    \begin{align*}
        x^* \in\pi_{A,b}(\mathcal{Y}_{A,b}) & \,  \Longleftrightarrow \, \text{there is } z\in Y_{A,\beta} \cap (\mathbb{C}^*)^n \text{ such that } (f_j)_{x = x^*, y = z} =0 \text{ for }j=1,\dotsc,n\\
        & \, \Longleftrightarrow \, \text{there is } z \in V(J_{x^*}) \text{ such that } z \in Y_{A,\beta} \cap (\mathbb{C}^*)^n\\
        & \, \Longleftrightarrow \,  \text{there is } z \in V(J_{x^*}) \text{ such that } g_1(z)=\cdots=g_r(z)=0.
    \end{align*}
    The ideal $J=\langle f_1,\dotsc,f_n\rangle \subset K[y^{\pm 1}]$ is zero-dimensional and radical.
    By \Cref{thm:EigenvalueEigenvector}, the eigenvalues of $M_{g_j}$ are given by $\{g_j(z): z\in V(J)\} \subset \overline{K}$.
    If there exists $z\in V(J_{x^*})$ such that $g_1(z)=\dotsc=g_r(z)=0$, then ${\rm ev}_z: \mathbb{C}[y^{\pm 1}]/J_{x^*} \rightarrow \mathbb{C}[y^{\pm 1}]/J_{x^*}, [h] \mapsto h(z)$ is a common left eigenvector of $(M_{g_1})_{|x = x^*},\dotsc,(M_{g_r})_{|x = x^*}$ with eigenvalue zero. Hence, ${\rm ev}_z$ is represented by a row vector of length $2^n$, which is a left kernel vector of the concatenated matrix $(M_{g_1} | \cdots | M_{g_r})_{|x = x^*}$. This shows the inclusion $\pi_{A,b}({\cal Y}_{A,b}) \subseteq W$, and hence ${\cal L}_{A,b} \subseteq W$.

    For the reverse inclusion, suppose that ${\rm rank}(M_{g_1}|\dotsc|M_{g_r})_{|x = x^*} <2^n$ for some~$x^* \in \mathbb{C}^n$. Then there exists a left null vector $v^t \in \mathbb{C}^{2^n}$. 
    Applying column operations to the matrix $(M_{g_1} | \cdots | M_{g_r})_{|x = x^*}$, we may replace each $g_j$ by a random $\mathbb{C}$-linear combination $\tilde{g}_j$ of $g_1, \ldots, g_r$. This has the effect that for each $z \in V(J_{x^*})$ and each $j = 1, \ldots, r$, we have $\tilde{g}_j(z) = 0$ if and only if $g_1(z) = \cdots = g_r(z) = 0$, i.e., $\tilde{g}_j(z) = 0 \Leftrightarrow z \in Y_{A,\beta} \cap (\mathbb{C}^*)^n$. Since $v^t \cdot (M_{\tilde{g}_1} | \cdots | M_{\tilde{g}_r})_{|x = x^*} = 0$, each of the matrices $M_{\tilde{g}_j}$ has a zero eigenvalue. Theorem \ref{thm:EigenvalueEigenvector} implies that for each $j$, $\tilde{g}_j(z^{(j)}) = 0$ for some $z^{(j)} \in V(J_{x^*})$. We conclude that $V(J_{x^*}) \cap Y_{A,\beta} \cap (\mathbb{C}^*)^n \neq \emptyset$ and $x^* \in \pi_{A,b}({\cal Y}_{A,b})$. 
\end{proof}

\begin{proposition}\label{prop:DetRepresentation}
    In the situation of Theorem \ref{thm:RankCondition}, if $A$ has rank $n-1$ and $I_{A,\beta} = \langle g \rangle$, then
    \begin{equation} \label{eq:detrep}
    {\cal L}_{A,b} =\{x\in\C^n : \det(M_{g})=0\}.
    \end{equation}
    Moreover, if the prime ideal of ${\cal L}_{A,b}$ is $\langle f \rangle \subset \mathbb{C}[x]$, then $\det(M_g) = c \,  f^{\deg \pi_{A,b}}$ for some~$c \in \mathbb{C}^*$. 
\end{proposition}

\begin{proof}
Equation \eqref{eq:detrep} is an immediate consequence of Theorem \ref{thm:RankCondition}. To compute the order of vanishing of $\det(M_g)$ along ${\cal L}_{A,b}$, we introduce a small parameter $t$ and evaluate $M_g(x)$ at the point $x^* + t \cdot x_0$ for a generic point $x^* \in {\cal L}_{A,b}$ and a generic point $x_0 \in \mathbb{C}^n$. By \eqref{eq:detrep} we~have
\begin{equation} \label{eq:aux} \det M_g(x^* + t \cdot x_0) \, = \,  c \, f(x^* + t \cdot x_0)^k  \, = \, c_1 \, t^k + O(t^{k+1})
\end{equation}
for some positive $k$ and $c_1 \in \mathbb{C}^*$. 
By Theorem \ref{thm:EigenvalueEigenvector}, the eigenvalues of $M_g(x^* + t \cdot x_0)$ are the values $g(z(t))$ for $z(t) \in V(J_{x^* + t \cdot x_0})$. Therefore, we have $\det  M_g(x^* + t \cdot x_0) = \prod_{z(t) \in V(J_{x^* + t \cdot x_0})} g(z(t))$. By genericity of $x^* \in {\cal L}_{A,b}$, there are $\deg \pi_{A,b}$ eigenvalues for which $g(z(0)) = 0$. Hence, to conclude that $\det  M_g(x^* + t \cdot x_0) = c_1 \, t^{\deg \pi_{A,b}} + O(t^{\deg \pi_{A,b} + 1})$, it remains to show that for each eigenvalue with $g(z(0)) = 0$, we have $g(z(t)) = \tilde{c}_1 t + O(t^2)$ for some $\tilde{c}_1 \neq 0$. 
The coordinates of $z(t)$ satisfy $1-2(x_j^*+tx_{0j})z_j(t)+z_j(t)^2 = 0$. Solving this explicitly yields $z_j(t)=z_j(0) + t \, x_{0j}(1 +x_j^*((x_j^*)^2-1)^{-1/2})+O(t^2)$.  By genericity of $x_0$, the curve parametrized by $t \mapsto z(t)$ intersects $Y_{A,\beta}$ transversally at $t = 0$. This implies that $g(z(t))$ has vanishing order one at $t = 0$. We conclude that $k = \deg \pi_{A,b}$ in \eqref{eq:aux}, as desired.
\end{proof}
\section{Kuramoto oscillators} \label{sec:4}

The Kuramoto model is a system of ordinary differential equations, widely used to describe systems of coupled phase oscillators \cite{kuramoto1975self}. In this section, we explain how Lissajous varieties of type ${\cal S}_A$ show up in the study of its steady states. 
Let $G=(V,E)$ be a graph representing the coupling of a system of $m$ oscillators, where $m=|V|$. Let $n = |E|$ be the number of edges and, for $k = 1, \ldots, m$, let $V_k$ be the set of vertices $v\in V$ adjacent to $v_k$. The \emph{Kuramoto model} is a system of $m$ ordinary differential equations in $m$ unknown functions $\theta_k\colon{\mathbb{R}}\rightarrow\mathbb{R}$:
\begin{equation}\label{eq:KuramotoModel}
\dot{\theta}_k \, = \,  \omega_k + \sum_{j\in V_k} K_{kj} \, \sin{(\theta_j - \theta_k)}, \quad k \, = \, 1, \ldots, m.
\end{equation}
The function $\theta_k(t)$ is the angle at vertex $k$ at time $t$, $\dot{\theta}_k$ is its derivative, $\omega_k \in \mathbb{R}$ is the natural frequency of the $k$-th oscillator, and $K_{kj} \in \mathbb{R}_+$ is the coupling strength of the edge $(k,j) \in E$.

\begin{example}\label{ex:KuramotoSysC3}
    Let $G=C_3$ be the triangle graph. Here $m=n=3$, and \Cref{eq:KuramotoModel} reads:
    \[
    \begin{aligned}
        \dot{\theta}_1 \, = \,  K_{12}\sin{(\theta_2 - \theta_1)} +K_{13}\sin{(\theta_3 - \theta_1)} + \omega_1, \\
        \dot{\theta}_2 \, = \,  K_{12}\sin{(\theta_1 - \theta_2)} +K_{23}\sin{(\theta_3 - \theta_2)} + \omega_2, \\
        \dot{\theta}_3 \, = \,  K_{13}\sin{(\theta_1 - \theta_3)} +K_{23}\sin{(\theta_2 - \theta_3)} + \omega_3.
    \end{aligned}
    \]
    The spring network associated with this graph is shown in \Cref{fig:springnetworkC3}. 
    This is a mechanical illustration of the Kuramoto model \cite{dorfler2014synchronization}. The vertices of the graph are constrained to lie on a circle and are connected by spring-like edges. 
    The Kuramoto model describes the angular velocity $\dot{\theta}_k$ of the particle at vertex $k$ as it moves around the circle. One sees from the equations that, for an equilibrium point to exist, we must have $\omega_1 + \omega_2 + \omega_3 = 0$.
    \begin{figure}[ht]
    \centering
    \noindent\makebox[\textwidth]{
    \begin{tikzpicture}[scale=1.7,
    spring/.style={decorate, decoration={aspect=0.3, segment length=3pt, amplitude=1.5pt, coil}}]

    \draw[blue, dashed] (0,0) circle (1);

     \foreach \i/\name in {225/1, 345/3, 105/2} {
        \node[circle, fill=black, inner sep=2pt, label={\i:\name}] (\name) at (\i:1) {};
    }

    \draw[spring, thick] (1) -- (3) node[midway, below, xshift=5pt] {$K_{13}$};
    \draw[spring, thick] (2) -- (3) node[midway, above right, xshift=-5pt] {$K_{23}$};
    \draw[spring, thick] (2) -- (1) node[midway, xshift=-12pt, yshift=5pt] {$K_{12}$};

    \draw[blue, thin, ->] (0,0) -- (1,0);
    \draw[blue, thin, ->] (0,0) -- (2);
    \draw[->, thin] (8:0.2) arc[start angle=-2, end angle=105, radius=0.2];
        \node at (65:0.35) {\scriptsize $\theta_2$};
        
    \end{tikzpicture}
    }
    \caption{The spring network of the triangle graph $C_3$.}
    \label{fig:springnetworkC3}
    \end{figure}
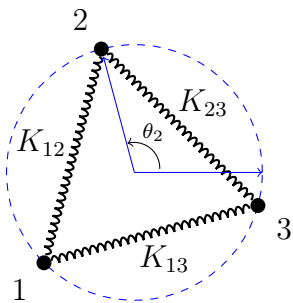
\end{example}

We are interested in the steady states of \Cref{eq:KuramotoModel}. That is, we want to solve the trigonometric equations $\dot{\theta}_k = 0$. For this we rephrase these equations as algebraic equations and use tools from computational algebraic geometry. This is much in the spirit of \cite{chen2018counting,harrington2023kuramoto,mehta2015algebraic}.

The approach taken in \cite{chen2018counting,harrington2023kuramoto,mehta2015algebraic} and references therein is to substitute $x_k = \sin(\theta_k)$ and $y_j = \cos(\theta_j)$. The formula 
$
\sin(\theta_j - \theta_k) = \sin(\theta_j)\cos(\theta_k) - \cos(\theta_j)\sin(\theta_k)
$
turns \eqref{eq:KuramotoModel} into
\begin{equation} \label{eq:bilinear}
f_k \, = \,  \omega_k + \sum_{v_j\in V_k} K_{kj} \, (x_jy_k - x_ky_j).
\end{equation}
Let $I_{\theta} = \langle x_1^2+y_1^2-1,\dotsc,x_m^2+y_m^2-1 \rangle \subset \mathbb{C}[x,y]$ and $I_G=\langle f_1,\dotsc,f_m \rangle \subset \C[x,y]$. In \cite{harrington2023kuramoto}, the \emph{Kuramoto ideal} is $I_K = I_{\theta} + I_G$ and the \emph{Kuramoto variety} is its vanishing set $V(I_K) \subseteq \mathbb{C}^{2m}$.

In this section we explore a different algebraic perspective. 
We assume that $G$ is a simple and connected graph and identify $V = [m]$.
Let $A_G\in \mathbb{Z}^{m\times n}$ be the incidence matrix of $G$, defined as follows. 
If $k \to j$ is the $l$-th edge of $G$ with respect to an arbitrary fixed orientation and an arbitrary fixed ordering of the edges, then the $l$-th column of $A_G$ is $a_l = (e_k-e_j)^T$, where $e_k$ is the $k$-th standard basis vector of $\mathbb{R}^n$.
Since $G$ is connected, the rank of $A_G$ is $m-1$.
The Lissajous variety ${\cal L}_{A_G,\mathbf{1}}$ depends only on the row span of $A_G$. Therefore, unless otherwise specified, we let $A(G)$ be the matrix obtained by removing the last row of the incidence matrix $A_G$, and we set $d=m-1$. We refer to $A(G)$ as the \emph{reduced incidence matrix} of $G$. Below, unless specified otherwise, we fix $G$ and write $A = A(G) \in \mathbb{Z}^{d \times n}$ for short. In particular, unless specified otherwise, ${\cal S}_A = {\cal L}_{A(G),{\bf 1}} = {\cal L}_{A_G,{\bf 1}}$. We write $\omega \in \mathbb{R}^d$ for the vector of natural frequencies after dropping $\omega_m$. For simplicity, we shall assume in what follows that all constants $K_{ij}$ are equal to $K \in \mathbb{R}_+$. All statements are easily generalized to arbitrary $K_{ij}$.

With this setup, studying the steady states of \Cref{eq:KuramotoModel} amounts to studying the intersection of the Lissajous variety ${\cal S}_A$ with an affine linear space. Indeed, after substituting $x_l = \sin(\theta_k-\theta_j)$ in \eqref{eq:KuramotoModel}, the resulting equations are affine-linear. We obtain
\begin{equation}\label{eq:LissajousKuramotoEqs}
    x\in {\cal S}_A \quad \text{and} \quad Ax=\omega/K.
\end{equation}
A steady state is recovered from $x$ satisfying \eqref{eq:LissajousKuramotoEqs} by computing the fiber $\phi_{A,\bf 1}^{-1}(x)$. A different approach is to set $v_j = e^{i \theta_j}$ and solve the following nonlinear equations in $(v_1, \ldots, v_d)$: 
\begin{equation} \label{eq:Kuramoto_v_equations} \tfrac{\omega_{j}}{K} + \tfrac{a_{j1}}{2i}( v^{a_1} -  v^{-a_1}) + \cdots + \tfrac{a_{jn}}{2i}( v^{a_n} - v^{-a_n}) \, = \, 0, \, \, \, j = 1, \ldots, d .\end{equation}
The relation between these two approaches was exploited in the proof of Theorem \ref{thm:degree}, where the solution sets were denoted by $S_x$ and $S_v$ respectively, and we have $\psi_{A,\mathbf{1}}(S_v) = S_x$.

Recall that the degree of ${\cal S}_A$ is given by the number of intersection points with a generic affine space of complementary dimension. 
The affine linear space $Ax=\omega/K$ has codimension $d$, which matches the dimension of ${\cal S}_A$. However, it is not generic, as both the variety and the affine space depend on the matrix $A$. 
On the other hand, if $Ax=\omega/K$ and ${\cal S}_A$ intersect in a finite number of isolated solutions, then this number is bounded above by the degree. The volume in Theorem \ref{thm:degree} appeared as a bound on the number of isolated equilibria in \cite{chen2018counting}.

\begin{example}\label{ex:6solutionsC3}
    Let $G=C_3$ be the triangle graph, as in Example \ref{ex:KuramotoSysC3}. 
    For cycle graphs, we choose the following edge ordering and orientation: $1 \rightarrow 2, \, 2 \rightarrow 3, \, \ldots, \,  n \rightarrow 1$. 
    In the case $n=3$, the reduced incidence matrix $A$ with respect to this ordering coincides with $A$ from Example \ref{ex:elliptopeintro}, hence yielding the same Lissajous variety ${\cal S}_A$. 
    Since ${\rm rank}(A)$ is $m-1$, one of the equations in \eqref{eq:KuramotoModel} can be dropped. For feasibility, we must have $\sum_{l=1}^m \omega_l = 0$. We set the coupling strength $K$ to $1$, and choose $\omega = (\frac{1}{10},\frac{1}{5})$. We also fix $\theta_m=0$ and use $\theta_1, \ldots, \theta_{m-1}= \theta_d$ as coordinates on ${\rm Row}(A)$. With these choices, \Cref{eq:LissajousKuramotoEqs} reads
    \[ x^4 + 4x^2y^2z^2 - 2x^2 y^2 - 2x^2 z^2 + y^4 - 2 y^2z^2 + z^4 \, = \, 
         10 x - 10 z -1 \, = \, 
         - 5 x + 5 y - 1 \, = \, 0.
    \]
    This system has six distinct real solutions, attaining the degree of ${\cal S}_A$. The $(x,y,z)$-coordinates of these solutions determine the sines of the steady state angles: $y = \sin(\theta_2), \,  z = - \sin(\theta_1)$. The cosines can be found by substituting these values appropriately in $f_k = 0$, where $f_k$ is as in \eqref{eq:bilinear}, and solve the resulting system of linear equations. Alternatively, one solves two Laurent polynomial equations in two unknowns given by \eqref{eq:Kuramoto_v_equations}, and finds $\theta_j = -i \log(v_j)$.
\end{example}

\begin{remark} \label{rem:cyclegraph}
    Let $G$ be the cycle graph $C_n$. With the choices of Example \ref{ex:6solutionsC3}, its (reduced) incidence matrix has the same row span as the matrix $A_n$ from \eqref{eq:cycleAn}. The cycle polynomial $\Gamma_n'$ is the defining equation of the Lissajous variety ${\cal C}_{A_n} = {\cal L}_{A_n, {\bf 0}}$. 
\end{remark}

We interpret the equations \eqref{eq:LissajousKuramotoEqs} from an optimization perspective. For this discussion, the matrix $A \in \mathbb{Z}^{d \times n}$ has rank $d$, and it does not necessarily come from a graph. We define a submanifold ${\cal S}_A^+ \subset {\cal S}_A$ as follows. Consider the $d$-dimensional convex polytope $P = {\rm Row}(A) \cap [-\tfrac{\pi}{2}, \tfrac{\pi}{2}]^n$. obtained by intersecting a hypercube with the row span of $A$. We set 
\begin{equation} \label{eq:SA+} {\cal S}_A^+ \, = \, \sin({\rm int}(P)) \, \subset \, {\cal S}_A. \end{equation}
That is, ${\cal S}_A^+$ is the image under $y \mapsto (\sin(y_1), \ldots, \sin(y_n))$ of the interior of $P \subset {\rm Row}(A)$. That map is an isomorphism of manifolds ${\rm int}(P) \simeq {\cal S}_A^+$ with inverse given by the coordinate-wise arcsine function ${\rm arcsin}\colon (-1,1)^n \rightarrow (-\tfrac{\pi}{2}, \tfrac{\pi}{2})^n, \,{\rm arcsin}(x_1, \ldots, x_n) = ({\rm arcsin}(x_1), \ldots, {\rm arcsin}(x_n))$.

\begin{theorem}\label{thm:convexoptimizationKuramoto}
    Let $A \in \mathbb{Z}^{d \times n}$ be of rank $d$. We have $x^* \in {\cal S}_A^+ \cap \{ A K x = \omega \}$ if and only if $x^*$ is the unique minimizer of the following convex optimization problem: 
    \begin{equation} \label{eq:optprob} {\rm minimize} \, \sum_{j = 1}^n \Big (x_j \, {\rm arcsin}(x_j) + \sqrt{1-x_j^2} \Big), \quad {\rm subject~to} \quad Ax = \omega/K \, \,  \text{and} \, \, x \in (-1,1)^n . \end{equation}
\end{theorem}
\begin{proof}
    The function $g(t) = t \, {\rm arcsin}(t) + \sqrt{1-t^2}$ is strictly convex on the open interval $(-1,1)$. Hence, the objective function $\sum_{j = 1}^n g(x_j)$ is strictly convex on the feasible region of our optimization problem. If a minimizer $x^* \in (-1,1)^n$ exists, then it is unique, and it satisfies the first order optimality conditions 
    \[ \frac{\partial \, {\rm Lag}}{\partial x_j} \, = \, \frac{\partial}{\partial x_j} \Big ( \sum_{j = 1}^n g(x_j) - \lambda^t (Ax - \omega/K) \Big ) \, = \, 0 \quad \text{and} \quad Ax = \omega/K.\]
    Here ${\rm Lag}$ is the Lagrangian and $\lambda = (\lambda_1, \ldots, \lambda_d)$ are the Lagrange multipliers. The derivative of $g(t)$ is ${\rm arcsin}(t)$. Hence, the equations coming from partial derivatives of ${\rm Lag}$ with respect to $x_j$ are equivalent to ${\rm arcsin}(x) \in {\rm Row}(A)$, which implies ${\rm arcsin}(x) \in {\rm int}(P)$. Taking the coordinate-wise sine on both sides, we see that this is equivalent to $x \in {\cal S}_A^+$.  
\end{proof}

Motivated by Theorem \ref{thm:convexoptimizationKuramoto}, we introduce the following notation for the $A$-projection of~${\cal S}_A^+$:
\[ \Omega^+ \, = \, \{ A \, x \in \mathbb{R}^d \, : \, x \in {\cal S}_A^+ \} \, = \, A ({\cal S}_{A}^+).\]
The optimization problem \eqref{eq:optprob} has a unique minimizer if and only if $\omega/K \in \Omega^+$. The projection $A: {\cal S}_A^+ \rightarrow \Omega^+$ is one-to-one, and the inverse is given by $\omega \mapsto {\cal S}_A^+ \cap \{Ax = \omega\}$.
\begin{example}
    We set $K = 1$ and use the matrix $A = \left (\begin{smallmatrix}
        1 & 0 & -1\\-1 & 1 & 0
    \end{smallmatrix} \right )$ associated with $G = C_3$. The manifold ${\cal S}_A^+$ is shown in blue in the left part Figure \ref{fig:SA+andOmega+}.
    \begin{figure}
    \centering
    \includegraphics[height = 4.5cm]{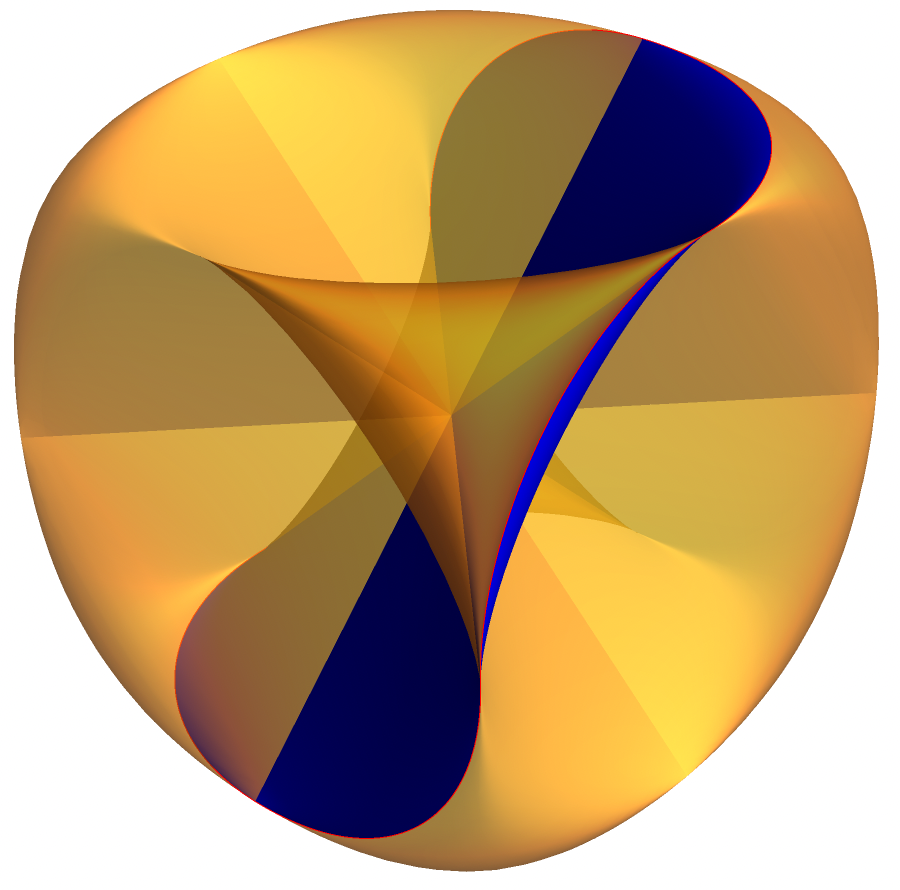} \quad \quad \quad \quad 
    \includegraphics[height = 4.5cm]{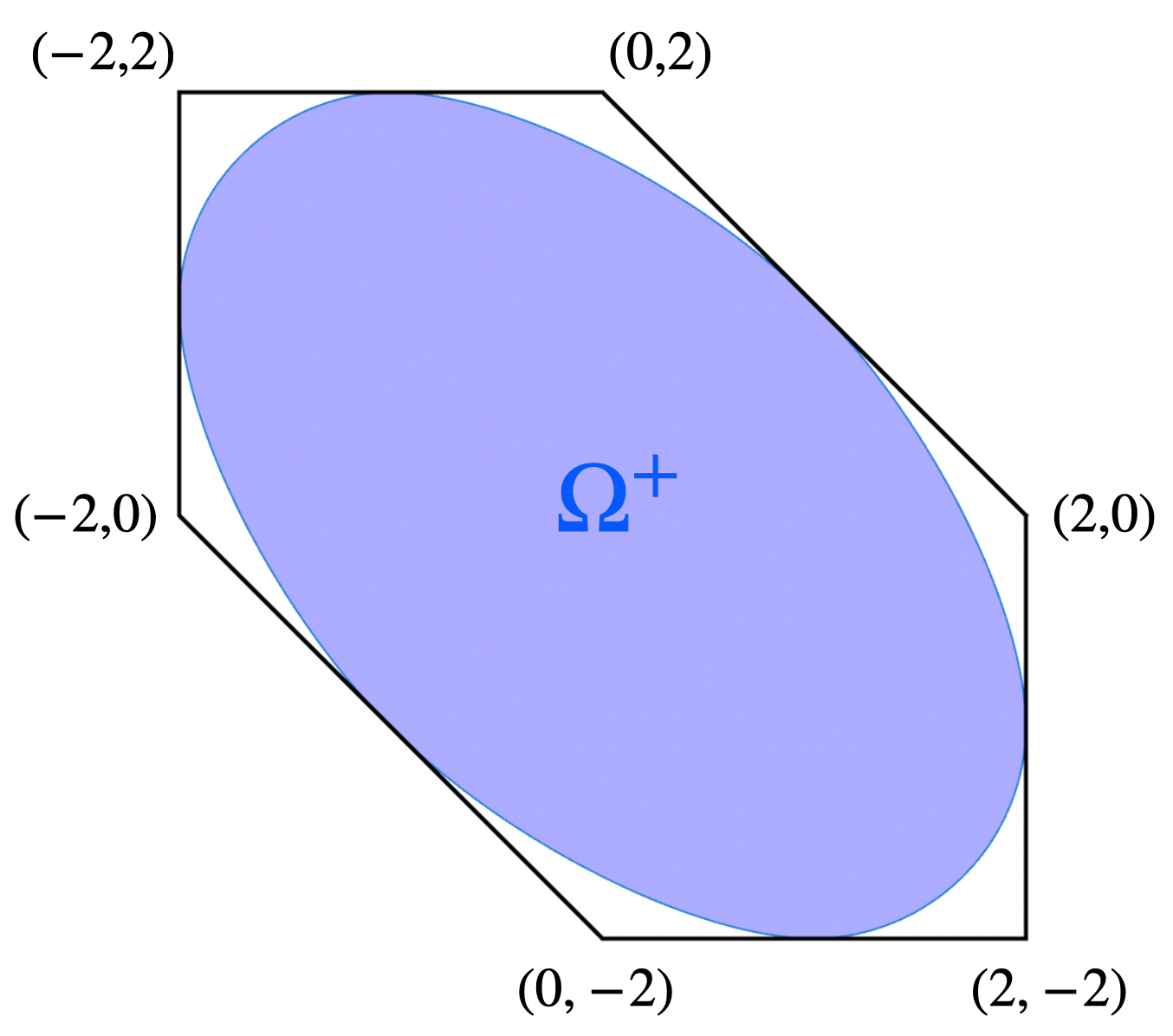}
    \caption{Positive regions for $x \in {\cal S}_A$ (left) and $\omega \in \mathbb{R}^2$ (right).}
    \label{fig:SA+andOmega+}
\end{figure}
    Its projection to $\mathbb{R}^2$ via $x \mapsto (x_1-x_3,-x_1 + x_2)$ is $\Omega^+$, seen in Figure \ref{fig:SA+andOmega+} (right). This set is contained in a hexagon $Q$, obtained as the $A$-projection of the cube $[-1,1]^3$. For $\omega \in {\rm int}(Q)$, the feasible region of \eqref{eq:optprob} is an open line segment $\ell = \{Ax = \omega\} \cap (-1,1)^n$. For $\omega \in Q \setminus \Omega^+$, the objective function does not attain a minimum on $\ell$. For $\omega \in \Omega^+$, there is a unique minimizer given by $\ell \cap {\cal S}_A^+$.   
\end{example}

We note that a different nonlinear (transcendental) optimization approach for $\omega = 0$ is discussed in \cite{ling2019landscape}. In Section \ref{sec:optim}, we shall establish a more general connection between Lissajous varieties and optimization. 
Before that, we conclude the present section with a remark on the stability of equilibria. 
Write the system of ODEs in \Cref{eq:KuramotoModel} as $\dot{\theta}=\Phi(\theta)$. 
Let $J_{\Phi}(\theta)=\begin{pmatrix}
    \frac{\partial\Phi_l}{\partial\theta_j}(\theta)
\end{pmatrix}_{lj}$ be the Jacobian matrix of $\Phi$. 
The matrix $J_{\Phi}(\theta)$ is symmetric, hence its eigenvalues are real. 
Because of the linear relation between the equations \eqref{eq:KuramotoModel} pointed out above, $J_{\Phi}(\theta)$ is rank deficient -- one of its eigenvalues is always zero.
A solution $\theta^*$ to $\Phi(\theta)=0$ is called \emph{linearly stable} if all other eigenvalues of $J_{\Phi}(\theta^*)$ are negative. If $\omega/K \in \Omega^+$, then the equilibrium $\theta^*$ corresponding to the minimizer $x^*$ from Theorem \ref{thm:convexoptimizationKuramoto} is linearly stable. 
\begin{proposition} \label{prop:stable}
    Let $A$ be the reduced incidence matrix of $G$ and let $\omega/K\in\Omega^+$. Then, the unique minimizer described in \Cref{thm:convexoptimizationKuramoto} yields a linearly stable equilibrium $\theta^*$ of the corresponding Kuramoto model. This is the unique vector $\theta^*$ satisfying $A^t \theta^* = {\rm arcsin}(x^*)$.
\end{proposition}
\begin{proof}
    If $\omega/K\in\Omega^+$, then $A^t \theta^*\in {\rm int}(P) \subset {\rm Row}(A)$. In particular, $\cos(a_l\cdot \theta^*)>0$ for all $l=1,\dotsc,n$. The statement then follows from \cite[Lemma 3.2]{harrington2023kuramoto}.
\end{proof}

While solving the equations \eqref{eq:Kuramoto_v_equations} leads to all potential equilibria, there may be many ``spurious'' solutions which have no physical meaning. Proposition \ref{prop:stable} says that if $\omega/K \in \Omega^+$, then a linearly stable solution can be found more efficiently via convex optimization. We now consider a larger graph to illustrate the computational range of this technique. 

\setcounter{MaxMatrixCols}{12}
\begin{example}
The graph shown in Figure \ref{fig:12edges} appears in \cite[Example 3.1]{harrington2023kuramoto}. The corresponding Lissajous variety ${\cal S}_A$ has degree 832. For $\omega = (0.1, \ldots, 0.1)$, the unique minimizer in \eqref{eq:optprob} is 
\[ x^* \, = \, (-0.050, 0.150, 0.083, 0.017, -0.050, 0.150, 0.017, 0.083, -0.067, 0.200, 0.067, 0.200),
\]
here rounded to three decimal places. This can be computed using, for instance, \texttt{JuMP.jl} \cite{Lubin2023}. The optimum is found in less than 0.01 seconds.  Solving $A^t \theta^* = {\rm arcsin}(x^*)$ we find that 
\[ \theta^* = (0.151, 0.285, 0.151, 0.285, 0.201, 0.268, 0.201), \]
which is a linearly stable equilibrium of the Kuramoto equations.
We verify that $x \in {\cal S}_A^+ \cap \{A \, x = \omega\}$ and $\omega \in \Omega^+$ by checking that $\omega = A \, {\rm sin}(A^t \theta^*)$, and $A^t \theta^* \in [-\tfrac{\pi}{2}, \tfrac{\pi}{2}]^{12}$.
\begin{figure}
\centering
\includegraphics[height=3.5cm]{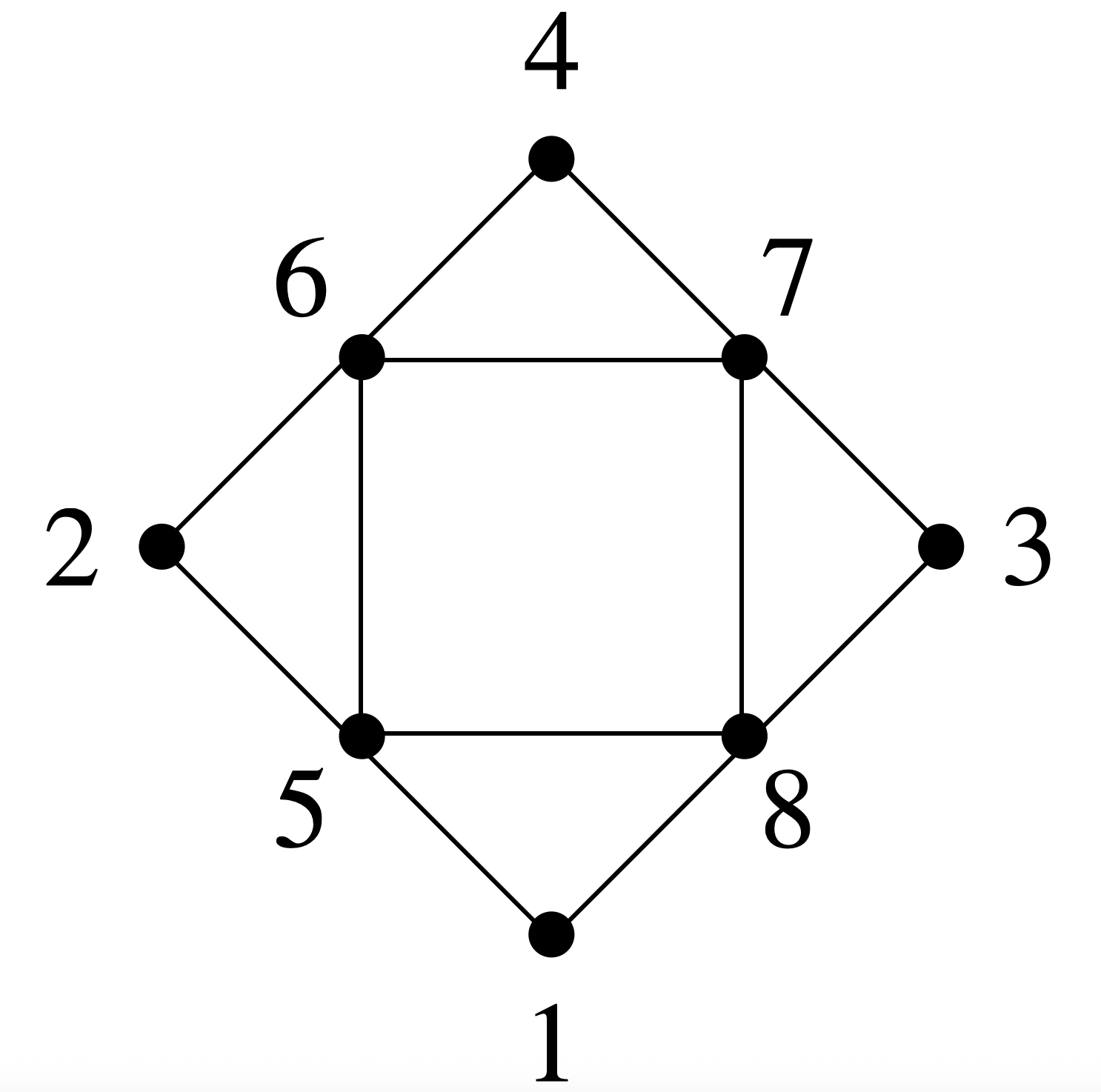} 
\quad 
\includegraphics[height=3.5cm]{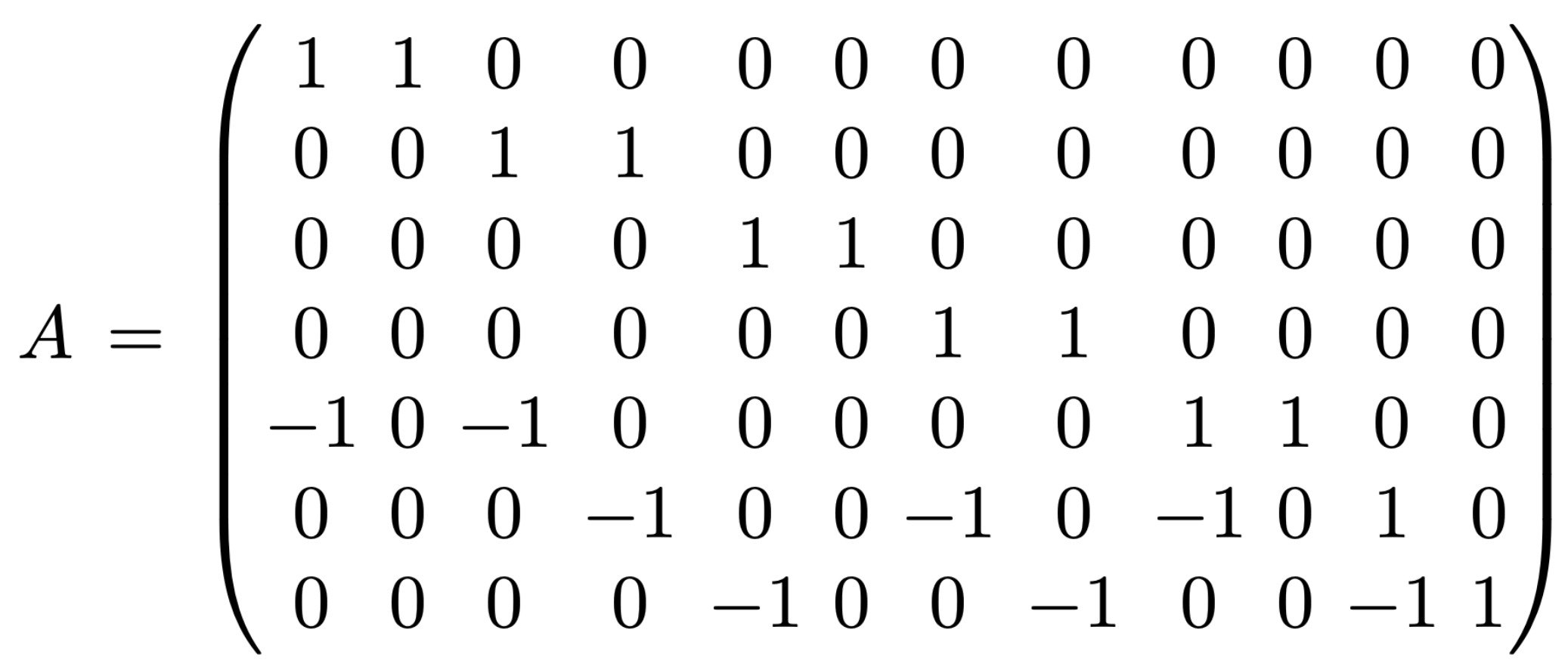}
    \caption{A graph and its reduced incidence matrix.}
    \label{fig:12edges}
\end{figure}    
\end{example}

\section{Positive points and convex optimization} \label{sec:optim}

In Theorem \ref{thm:convexoptimizationKuramoto} we showed that the open subset ${\cal S}_A^+ \subset {\cal S}_A = {\cal L}_{A,\bf 1}$ parametrizes all solutions of the optimization problem \eqref{eq:optprob} for $\omega \in \Omega^+$. In this section we prove a similar characterization for an arbitrary Lissajous variety ${\cal L}_{A,b}$. In particular, we generalize Theorem \ref{thm:convexoptimizationKuramoto} and Proposition \ref{prop:stable}. The Kuramoto model is generalized by the dynamical system 
\begin{equation} \label{eq:ourODE} \dot \theta \, = \, A \, \phi_{A,b}(\theta) - \omega, \end{equation}
where $A \in \mathbb{Q}^{d \times n}$ has rank $d$, $b \in \mathbb{R}^n$, and $\omega \in \mathbb{R}^d$. The map $\phi_{A,b}:[-\pi, \pi]^d \rightarrow [1,1]^n$ is 
\[ \phi_{A,b}(\theta) = (\cos(a_1 \cdot \theta - b_1 \, \tfrac{\pi}{2}), \ldots, \cos(a_1 \cdot \theta - b_n \, \tfrac{\pi }{2})).\]
Investigating the equilibria of this dynamical system leads us to solve $A \, \phi_{A,b}(\theta) - \omega = 0$. We complexify $\phi_{A,b}: \mathbb{C}^d \rightarrow \mathbb{C}^n$ and refer to the solutions $\theta \in \mathbb{C}^d$ as \emph{steady states} or \emph{equilibria}. The steady states correspond to points in the intersection ${\cal L}_{A,b} \cap \{A x = \omega \}$. We shall define ${\cal L}_{A,b}^+ \subset {\cal L}_{A,b} \cap (-1,1)^n$ such that there is at most one intersection point in ${\cal L}_{A,b}^+ \cap \{A x = \omega \}$. Moreover, if such a point exists, then it is the solution to a convex optimization problem.

The affine linear space $L_{A,b}$ is defined as $L_{A,b} = {\rm Row}(A) - \tfrac{b \pi}{2}$. Its image under the coordinate-wise cosine is ${\cal L}_{A,b}$. The appropriate generalizations of ${\cal S}_A^+$ and $\Omega^+$ are as follows: 
\[ {\cal L}_{A,b}^+ \, = \, \cos (L_{A,b} \cap (0,\pi)^n) \, \subset \, {\cal L}_{A,b},\quad \quad \Omega^+_{A,b} \, = \, \{ A \, x \in \mathbb{R}^d \, : \, x \in {\cal L}_{A,b}^+ \} \, = \, A({\cal L}_{A,b}^+).\]
One checks that this is consistent with the previous section, in that ${\cal L}_{A,\bf 1}^+ = {\cal S}_A^+$ and $\Omega_{A,{\bf 1}}^+ = \Omega^+$. 

Restricting to the real points of $L_{A,b}$, we have that ${\rm cos}(L_{A,b} \cap \mathbb{R}^n) = \cos (L_{A,b} \cap [-\pi,\pi]^n)$. This is a subset of ${\cal L}_{A,b}(\mathbb{R})$. The `$+$' in our notation is motivated by the fact that ${\cal L}_{A,b}^+$ is the image under the cosine map of all ``positive'' tuples of angles in $L_{A,b} \cap (0,\pi)^n$. 

\begin{theorem}\label{thm:convexoptimization}
    Let $A \in \mathbb{Q}^{d \times n}$ be of rank $d$. We have $x^* \in {\cal L}_{A,b}^+ \cap \{ A x = \omega \}$ if and only if $x^*$ is the unique minimizer of the following convex optimization problem: 
    \begin{equation} \label{eq:optprob2} {\rm minimize} \, \,  \frac{\text{-}\pi}{2} b^t x -  \sum_{j = 1}^n \Big (x_j \, {\rm arccos}(x_j) - \sqrt{1-x_j^2} \Big), \quad {\rm s.\,t.}  \quad Ax = \omega \, \,  \text{and} \, \, x \in (-1,1)^n . \end{equation}
    In particular, a minimizer exists if and only if $\omega \in \Omega_{A,b}^+$, and in that case it is unique. 
\end{theorem}
\begin{proof}
The first and second order derivatives of  $g(t) = - t \, {\rm arccos}(t) +  \sqrt{1-t^2}$ are
\[ g'(t) \, =\, - {\rm arccos}(t) \quad \text{and} \quad g''(t) = (1-t^2)^{-1/2}.\]
Hence, the objective function in \eqref{eq:optprob2} is strictly convex on $(-1,1)^n$, and so is its restriction to the feasible region. The rest of the proof is analogous to that of Theorem \ref{thm:convexoptimizationKuramoto}. The method of Lagrange multipliers gives the first order optimality conditions ${\rm arccos}(x) \in {\rm Row}(A) - \tfrac{\pi b}{2}$ and $Ax = \omega$, which is equivalent to $x \in {\cal L}_{A,b}^+ \cap \{Ax = \omega\}$.
\end{proof}

\begin{remark}
    Theorem \ref{thm:convexoptimization} is inspired by analogous convex optimization problems in which the constraint $x \in (-1,1)^n$ is replaced by $x \in \mathbb{R}^n_+$ and the objective function is a strictly convex function $G(x)$ on the positive orthant. Often $G$ is of the form $G(x) = \sum_{j = 1}^n g(x_j)$. Appropriate choices of $g$ lead naturally to semi-algebraic descriptions of the unique minimizer, similar to $x^* = {\cal L}_{A,b} \cap \{Ax = \omega\}$ in Theorem \ref{thm:convexoptimization}. For $g(t) = \log(t)$, the Lissajous variety is replaced by a \emph{reciprocal linear space} \cite{de2012central}. The function $g(t) = t \log (t) - t$ naturally leads to \emph{positive toric varieties} \cite{sturmfels2024toric}. If the universal barrier function $G(x)$ of the feasible polytope is minimized instead, then one intersects $\{A x = \omega\}$ with the \emph{Santal\'o patchwork} \cite{pavlov2025santalo}.
\end{remark}

\begin{remark}
    The coordinates of the minimizer $x^*$ of \eqref{eq:optprob2} are algebraic functions of $\omega$. Their minimal polynomial in $\mathbb{Q}(\omega)[x_j]$ has degree at most $\deg( {\cal L}_{A,b})$, see the formula stated in Theorem \ref{thm:degree}. This follows from the fact that $x^* \in {\cal L}_{A,b} \cap \{A x = \omega\}$ by Theorem \ref{thm:convexoptimization}.  
\end{remark}

Write the equations \eqref{eq:ourODE} as $\dot \theta = \Phi_{A,b}(\theta)$ and let $J_{\Phi_{A,b}}(\theta)$ be the $d \times d$ Jacobian matrix. We say that a steady state solution $\theta^*$ is \emph{linearly stable} if all eigenvalues of $J_{\Phi}(\theta^*)$ are negative. 
\begin{proposition}
Let $\omega \in \Omega_{A,b}^+$ and let $x^* \in {\cal L}_{A,b}^+$ be the unique  minimizer of the optimization problem \eqref{eq:optprob2}. The unique solution $\theta^*$ of the linear equations $A^t \theta^* = {\rm arccos}(x^*) + \tfrac{\pi b}{2}$ is a linearly stable steady state solution of the dynamical system \eqref{eq:ourODE}. 
\end{proposition}

\begin{proof}
Note that $\theta^*$ is a steady state solution by construction: $A \phi_{A,b}(\theta^*) = A \, x^* = \omega$.  The Jacobian matrix $J_{\Phi_{A,b}}$ has the following explicit expression: 
\[ J_{\Phi_{A,b}}(\theta) \, = \, - A \, \,  {\rm diag} (\sin(a_1 \cdot \theta - b_1 \tfrac{\pi}{2}), \ldots, \sin(a_n \cdot \theta - b_n \tfrac{\pi }{2})) \, \,  A^t. \]
Since ${\rm arccos}(x^*) = A^t \theta^* - \frac{\pi b}{2}\in (0,\pi)^n$, the diagonal matrix in this expression has positive diagonal entries for $\theta = \theta^*$. Hence, $J_{\Phi_{A,b}}(\theta^*)$ is negative definite. 
\end{proof}

\begin{example} \label{ex:dynsyscircle}
The Lissajous curve ${\cal L}_{A,b}$ for the data $A = \begin{pmatrix}
1 & 1 
\end{pmatrix}$, $b = (0,1)$ is the circle with defining equation $x^2 + y^2 = 1$, see Example \ref{ex:circleintro}. The semi-algebraic set ${\cal L}_{A,b}^+$ is the segment of the curve contained in the $(-,+)$ quadrant, see Figure \ref{fig:circle}. 
\begin{figure}
\centering
\includegraphics[height = 4cm]{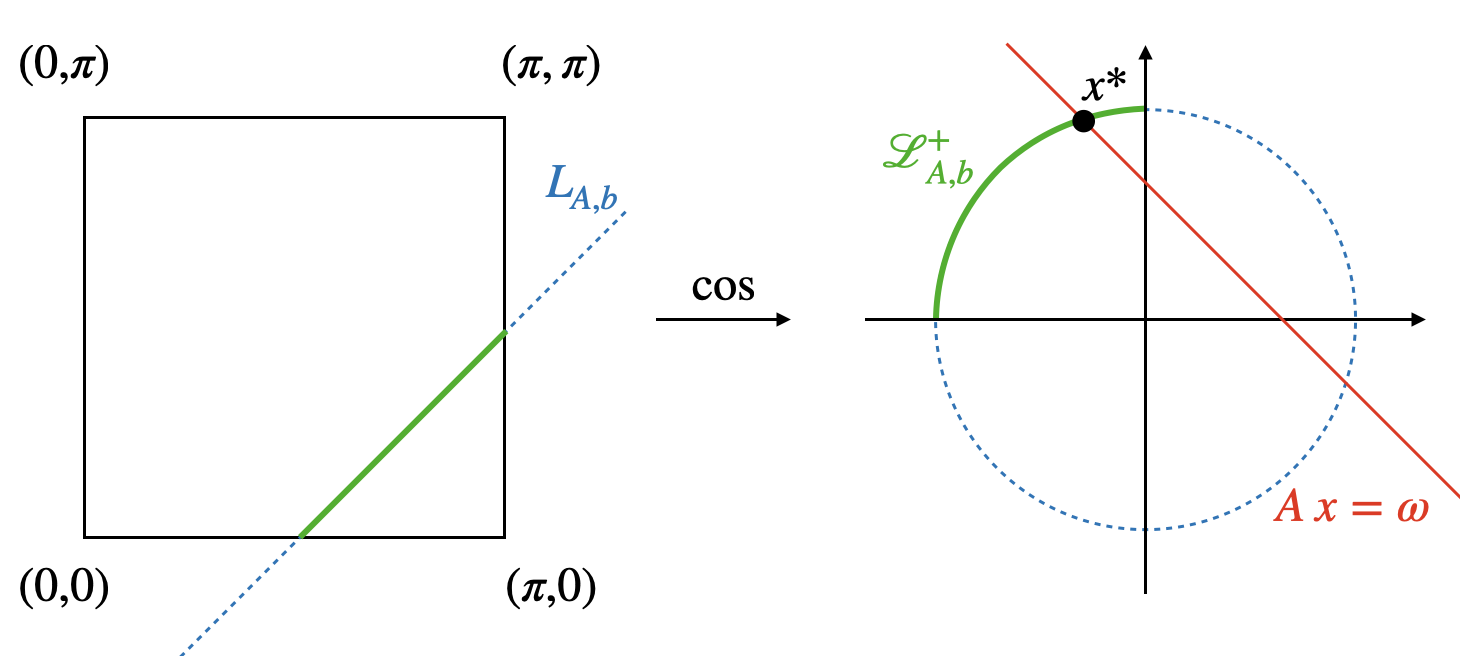}
\caption{The line $x + y = \omega$ has precisely one intersection point with ${\cal L}_{A,b}^+$ for $\omega \in (-1,1)$. }
\label{fig:circle}
\end{figure}
The projection $\Omega_{A,b}^+ = A({\cal L}_{A,b}^+)$ is the open line segment $(-1,1)$. For $\omega \in \Omega_{A,b}^+$, the line $x + y = \omega$ has one intersection point with ${\cal L}_{A,b}^+$. This is the unique solution to \eqref{eq:optprob2}. The differential equation 
\[ \dot \theta \, = \, \cos(\theta) + \sin(\theta) - \omega, \quad \omega \in \Omega_{A,b}^+ \]
has two steady state solutions $\theta^*$. Only one of them is stable. For concreteness, set $\omega = 0.6$. The stable equilibrium is $\theta^* \approx 1.918$. Its image $(\cos(\theta^*), \sin(\theta^*))$ is ${\cal L}_{A,b}^+ \cap \{ x + y = 0.6\}$.
\end{example}

\section{Lissajous discriminants} \label{sec:5}

Varying the natural frequencies $\omega$ affects the number of real solutions to the systems of polynomial equations \eqref{eq:LissajousKuramotoEqs} and \eqref{eq:Kuramoto_v_equations}. In the context of dynamical systems, varying $\omega$ influences the nature of the equilibria, which is the topic of bifurcation analysis.  Investigating this leads us to study a discriminant locus in the $\omega$ parameters. More precisely, we shall define a variety $\nabla_{A,b} \subset \mathbb{C}^d$, which is expected to be a hypersurface, such that the number of real solutions to 
\begin{equation}  \label{eq:kuramoto_v_eqs_withbeta}
\tfrac{a_{j1}}{2}(\beta_1 v^{a_1} + \beta_1^{-1} v^{-a_1}) + \cdots + \tfrac{a_{jn}}{2}(\beta_n v^{a_n} +  \beta_n^{-1} v^{-a_n}) - \omega_j \, = \,  0, \quad j = 1, \ldots, d \end{equation}
is constant for $\omega$ in each connected component of $\mathbb{R}^d \setminus \nabla_{A,b}$. For $b = \mathbf{1}$ ($\beta =  -i \cdot \mathbf{1}$), these are the equations \eqref{eq:Kuramoto_v_equations}. Taking cues from standard discriminant analysis, $\nabla_{A,b}$ should consist of points $\omega \in \mathbb{C}^d$ for which two complex solutions of \eqref{eq:kuramoto_v_eqs_withbeta} collide. We now make this precise.

We continue to assume that the matrix $A \in \mathbb{Z}^{d \times n}$ has full rank $d$. Note that we can write the equations \eqref{eq:kuramoto_v_eqs_withbeta} in a compact way as follows: $A  \, \psi_{A,b}(v)= \omega$, where $\psi_{A,b}$ is as in \eqref{eq:psiAb}. Consider the \emph{incidence variety} $W_{A,b} = \{(v, \omega) \in (\mathbb{C}^*)^d \times \mathbb{C}^d \, : \, A \, \psi_{A,b}(v)= \omega\}$. The fiber of 
\begin{equation} \label{eq:promega} {\rm pr}_\omega \, : \, W_{A,b} \, \longrightarrow \,  \mathbb{C}^d, \quad (v, \omega) \longmapsto \omega\end{equation}
over $\omega \in \mathbb{C}^d$ consists of the solutions $v$ to \eqref{eq:kuramoto_v_eqs_withbeta} for that fixed value of $\omega$. The \emph{toric Jacobi matrix} of $v \mapsto A \, \psi_{A,b}(v)$ is the $d \times d$ matrix given by 
\[ J_{A,b}(v) \, = \, \left( v_j \, \frac{\partial}{\partial v_j} (A \, \psi_{A,b}(v))_k\right)_{1 \leq j,k \leq d}  \, = \, \frac{1}{2} \,  A \, {\rm diag}( \beta_1 v^{a_1}-\beta_1^{-1} v^{-a_1}, \ldots,\beta_n v^{a_n}-\beta_n^{-1} v^{-a_n} )\, A^t  .\]
This is the usual Jacobi matrix, with $\tfrac{\partial}{\partial v_j}$ replaced by the Euler operator $v_j \, \tfrac{\partial}{\partial v_j}$. One checks that for a point $v \in {\rm pr}_\omega^{-1}(\omega)$, the toric Jacobian determinant $\det J_{A,b}(v)$ vanishes if and only if the usual Jacobian determinant vanishes. We prefer the toric version because of the elegant expression $J_{A,b}(v) = \tfrac{1}{2}AD(v)A^t$, where $D(v)$ is the diagonal $n \times n$ matrix shown above. 
\begin{lemma}
    The toric Jacobian $\det J_{A,b}(v)$ is not identically zero as a Laurent polynomial in $\beta$ and $v$. Moreover, for generic $\beta \in (\mathbb{C}^*)^n$ and generic $\omega \in \mathbb{C}^d$, the fiber ${\rm pr}_\omega^{-1}(\omega)$ is finite. That is, for generic $\beta,\omega$ the equations \eqref{eq:kuramoto_v_eqs_withbeta} have finitely many solutions $v \in (\mathbb{C}^*)^d$. 
\end{lemma}
\begin{proof}
    For $v = {\bf 1} \in (\mathbb{C}^*)^d$ and $\beta = -i \cdot {\bf 1} \in (\mathbb{C}^*)^n$, we have $\det J_{A, {\bf 1} }({\bf 1}) = \det(-i \, A \, A^t) \neq 0$ because $A$ has full rank. This shows the first statement. For the second part of the lemma, since $\beta$ is generic we may assume that $\det J_{A,b}(v)$ is not identically zero as a Laurent polynomial in $v$. Pick $v_0 \in (\mathbb{C}^*)^d$ such that $\det J_{A,b}(v_0) \neq 0$ and let $\omega_0 = A \psi_{A,b}(v_0)$. By construction, $v_0$ is isolated in ${\rm pr}_\omega^{-1}(\omega_0)$. Since ${\rm pr}_{\omega}$ is a dominant morphism of irreducible $d$-dimensional varieties, this implies that its generic fiber is finite \cite[Chapter 1, \S 8, Theorem 2 and Corollary 1]{mumford1999red}.  
\end{proof}

\begin{remark} For fixed $b$, the toric Jacobian $\det J_{A,b}$ might be identically zero. This happens for $A = \begin{pmatrix} 1 & 1 \end{pmatrix}$ when $\beta_1 + \beta_2 = 0$. The equations \eqref{eq:kuramoto_v_eqs_withbeta} have no solutions for generic $\omega$. 
\end{remark}

\begin{definition} \label{def:lissajousdisc}
    The \emph{ramification locus} $R_{A,b} \subseteq W_{A,b}$ is the divisor 
    \[ R_{A, b} \, = \, \{(v, \omega) \in W_{A,b} \, : \, \det J_{A,b}(v) = 0 \}. \]
    The \emph{Lissajous discriminant} $\nabla_{A,b}$ is the associated branch locus: $\nabla_{A,b} = \overline{{\rm pr}_\omega(R_{A,b})} \subseteq \mathbb{C}^d$. If $\nabla_{A,b}$ is a hypersurface, then its defining equation is denoted by $\Delta_{A,b} \in \mathbb{C}[\omega_1, \ldots, \omega_d]$. If $\nabla_{A,b}$ has codimension greater than one, then we set $\Delta_{A,b} = 1$. 
\end{definition}
Note that the polynomial $\Delta_{A,b}$ is defined up to scaling by a nonzero complex number. 

\begin{remark}
    The Lissajous discriminant can be viewed as the branch locus of the linear projection of the Lissajous variety ${\cal L}_{A,b}$ given by the matrix $A$. It is an analog of the \emph{entropic discriminant} \cite{sanyal2013entropic}, for which ${\cal L}_{A,b}$ is replaced by the
     reciprocal linear space of ${\rm Row}(A)$. 
\end{remark}

By definition, the Lissajous discriminant is the variety of the elimination ideal
\begin{equation}\label{eq:computediscriminant}
    \langle \, A \, \psi_{A,b}(v)\, -\, \omega,\, \det(J_{A, b})\, \rangle\cap\C[\omega].
\end{equation}
Here we start from an ideal with $d + 1$ generators in $\mathbb{C}[v^{\pm 1}, \omega]$. We compute three examples.
\begin{example} \label{ex:discrunning}
    For $A=\begin{pmatrix} 1 & 2 \end{pmatrix} \in \mathbb{Z}^{1 \times 2}$, \Cref{eq:computediscriminant} reads
    \[
    \langle \beta v+\beta^{-1}v^{-1}+2(\beta^2v^2+\beta^{-2}v^{-2})-2 \, \omega, \beta v-\beta^{-1}v^{-1}+4(\beta^2v^2-\beta^{-2}v^{-2})\rangle \cap \C[\omega].
    \]
    Setting $b = 1$, $\beta=-i$, we compute $\Delta_{A,1}=256\omega^4 - 2367\omega^2 + 3375$, which has four real roots. A root $\omega^{(j)}$ of $\Delta_{A,1}$ corresponds to a tangent line $x + 2y = \omega^{(j)}$ of ${\cal S}_A = {\cal L}_{A,1}$; see Figure \ref{fig:4tg8fig}.~For $b=0$, the discriminant is $\Delta_{A, 0}=16\omega^3 - 31\omega^2 - 84\omega + 99 =(\omega - 1) (\omega - 3) (16\omega + 33)$.
    \begin{figure}
        \centering        \includegraphics[height = 4.5cm]{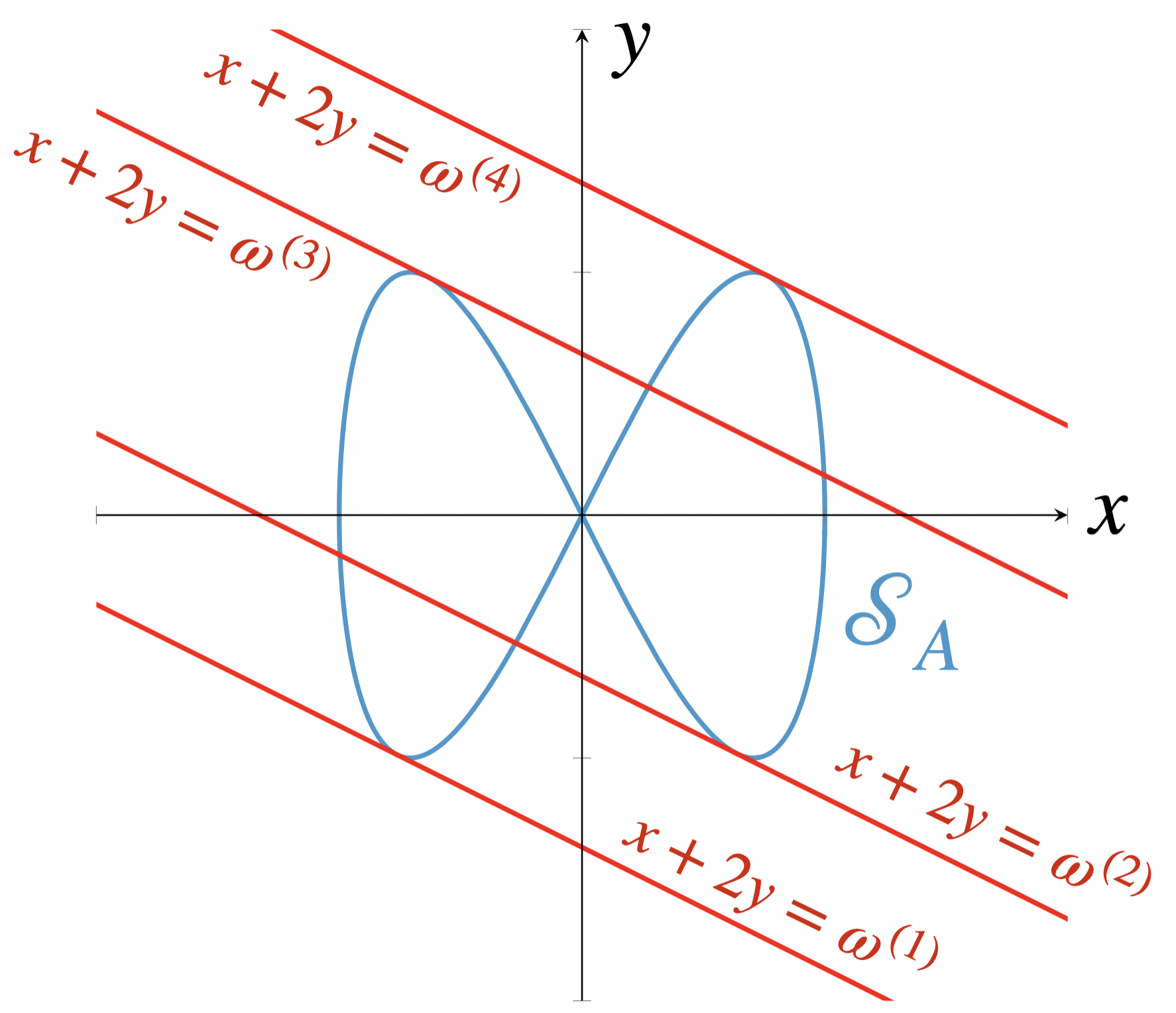}
        \caption{When $A=\begin{pmatrix} 1 & 2 \end{pmatrix}$, each of the four real roots $\omega^{(1)},\dotsc,\omega^{(4)}$ of $\Delta_{A,1}$ corresponds to a tangent line $x + 2y = \omega^{(j)}$ (in red) of ${\cal S}_A$ (in blue).}
        \label{fig:4tg8fig}
    \end{figure}
\end{example}

\begin{example} \label{ex:bifur}
    As mentioned above, studying the Lissajous discriminant of the matrix $A$ is closely related to the bifurcation analysis of the dynamical system \eqref{eq:ourODE}. Indeed, in view of linear stability analysis, an eigenvalue of the Jacobian matrix evaluated at a critical point can only change sign if its determinant vanishes. This happens when $\omega$ lies on the discriminant. For instance, the ramification locus $R_{A,b}$ in Example \ref{ex:dynsyscircle} is defined by the equations 
    \[ (v+v^{-1}) - i(v-v^{-1}) - 2 \, \omega \, = \, v-v^{-1} - i(v+v^{-1}) \, = \, 0, \]
    from which we see that $\nabla_{A,b} = \{ \pm \sqrt{2} \}$. These discriminant points correspond to the two red lines $\{ x + y = \pm \sqrt{2}\}$ tangent to the circle ${\cal L}_{A,b}(\mathbb{R})$, see Figure \ref{fig:circle}. The bounded discriminant chamber $(- \sqrt{2}, \sqrt{2})$ contains ${\cal L}_{A,b}^+ = (-1,1)$, and for $\omega$ in that chamber the dynamical system has one stable and one unstable equilibrium. For $\omega^2 > 2$, the system is unstable. 
\end{example}

\begin{example}\label{ex:C3discriminants}
    Let $A$ be as in Example \ref{ex:elliptopeintro}. 
    We compute the Lissajous discriminants of ${\cal L}_{A,b}$ for $b=\bf 0$ and $b=\bf1$. The curve $\nabla_{A,\bf 0}$ has degree $6$, and $\nabla_{A,\bf 1}$ has degree $12$. The equations~are 
    \begin{align*}
    \Delta_{A,\bf 0} = &-8\, \omega_1^5 + 4\, \omega_1^4\, \omega_2^2 - 20\, \omega_1^4\, \omega_2 - 23\, \omega_1^4 + 8\, \omega_1^3\, \omega_2^3 - 8\, \omega_1^3\, \omega_2^2 - 46\, \omega_1^3\, \omega_2 + 4\, \omega_1^3 + 4\, \omega_1^2\, \omega_2^4 \\ & + 8\, \omega_1^2\, \omega_2^3 - 69\, \omega_1^2\, \omega_2^2 + 6\, \omega_1^2\, \omega_2 + 36\, \omega_1^2 + 20\, \omega_1\, \omega_2^4 - 46\, \omega_1\, \omega_2^3 - 6\, \omega_1\, \omega_2^2 + 36\, \omega_1\, \omega_2 \\ & + 8\, \omega_2^5 - 23\, \omega_2^4 - 4\, \omega_2^3 + 36\, \omega_2^2,\\
    \Delta_{A,\bf 1} = & \, \, 64\, e_2^5 + 399 e_2^4 + 840\, e_2^3 + 376\, e_2^2e_3^2 + 766\, e_2^2 + 3056\, e_2 e_3^2 + 288\, e_2 
  - 16 e_3^4 + 5812\, e_3^2 + 27,
  \end{align*}
  where $e_2 = \omega_1\omega_2 + \omega_1 \omega_3 + \omega_2 \omega_3$, $e_3 = \omega_1\omega_2 \omega_3$ are the elementary symmetric polynomials and $\omega_3 = -\omega_1 - \omega_2$. We will explain the symmetric structure of $\Delta_{A,\bf 1}$ in Proposition \ref{prop:symmetry}. 
    Figure \ref{fig:C3discriminant} shows our two discriminant curves. They are the branch loci of the projection of Figure \ref{fig:C3} given by $A$. The reference \cite[Section 5.3]{bernal2023machine} studies  $\nabla_{A,\bf 1}$ via machine learning techniques. 

    Different regions in Figure \ref{fig:C3discriminant} correspond to different numbers of real solutions to \eqref{eq:kuramoto_v_eqs_withbeta}, or different numbers of real intersection points in ${\cal L}_{A,b} \cap \{Ax = \omega\}$.  For $\omega$ in the connected component of $\mathbb{R}^2 \setminus \nabla_{A,\bf 0}$ highlighted in orange, the line $Ax=\omega$ intersects ${\cal C}_A$ in three real points. In all other connected components, there is only one real intersection point. On the right, the green, blue and red components correspond to $\omega$ with six, four and two real intersection points respectively. We note that the convex hull of the real points of $\nabla_{A,\bf 1}$ contains the blue region $\Omega^+$ in Figure \ref{fig:SA+andOmega+} (right), but the two do not coincide. Figure \ref{fig:comparedisc} shows a comparison. This is similar to the inclusion $ (-1,1) \subsetneq (-\sqrt{2},\sqrt{2})$ in Example \ref{ex:bifur}. 
    \begin{figure}[h!]
    \centering
    \includegraphics[height = 4.5cm]{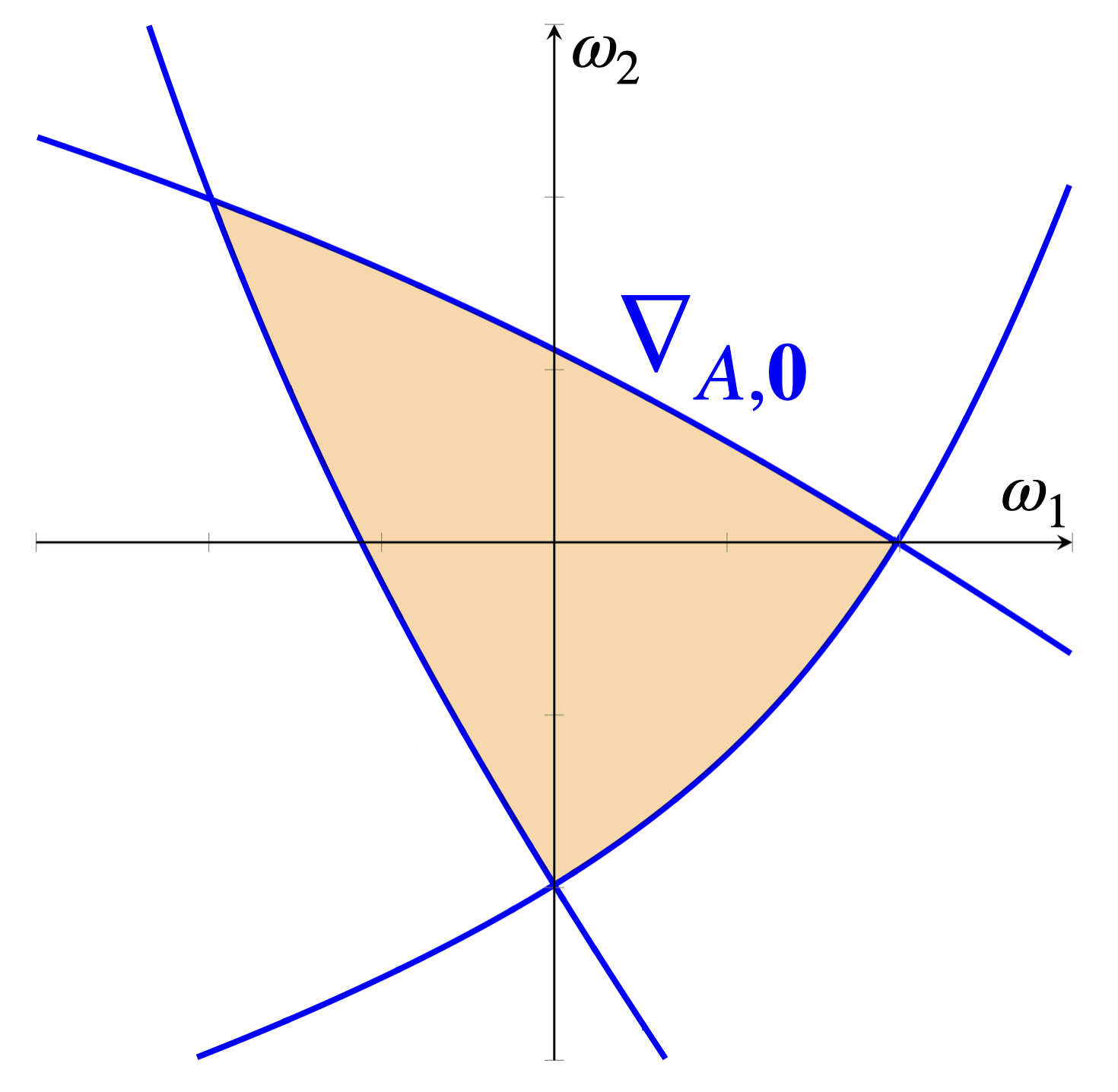} \quad \quad \quad 
    \includegraphics[height = 4.5cm]{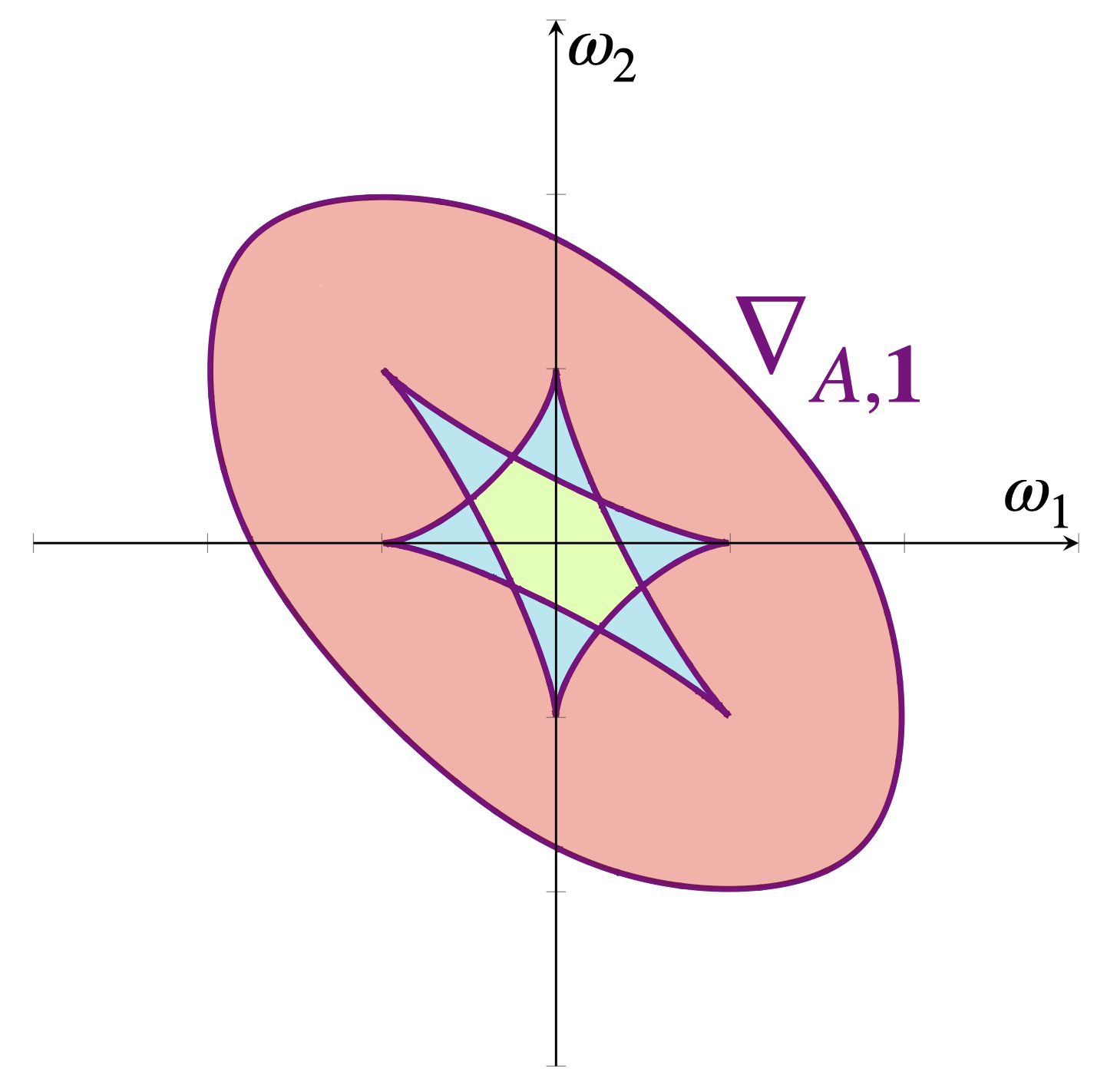}
    \caption{Lissajous discriminants from Example \ref{ex:discrunning}. }
    \label{fig:C3discriminant}
    \end{figure}

    \begin{figure}
        \centering
        \includegraphics[width=0.3\linewidth]{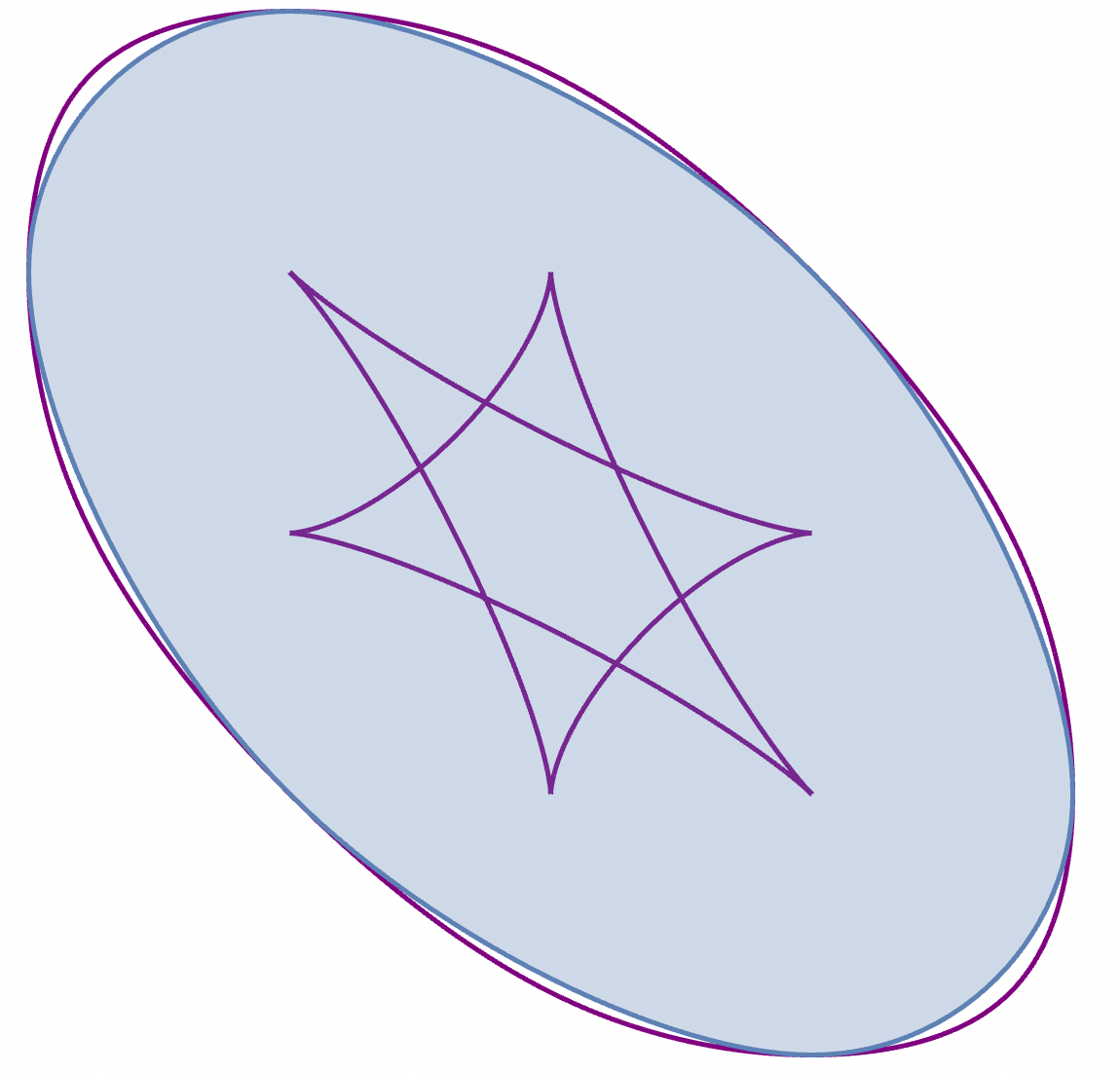}
        \caption{Comparing $\Omega^+$ (blue) with the Lissajous discriminant (purple).}
        \label{fig:comparedisc}
    \end{figure}
\end{example}

\begin{theorem} \label{thm:degdisc}
    If the Lissajous discriminant variety $\nabla_{A,b}$ is a hypersurface, then it has degree at most $(\deg {\rm pr}_\omega)^{-1} \cdot d \cdot d! \, {\rm vol}(P_A)$, where ${\rm pr}_\omega$ denotes the restriction $({\rm pr}_{\omega})_{|R_{A,b}}$. The number ${\rm vol}(P_A)$ is as in Theorem \ref{thm:degree}.
\end{theorem}

\begin{proof}
    For generic $\omega_0, \omega_1 \in \mathbb{C}^d$, the degree of $\nabla_{A,b}$ is the number of intersection points of $\nabla_{A,b}$ with the line parametrized by $\omega_0 + t \, \omega_1$. This is at most the number of solutions $(v,t) \in (\mathbb{C}^*)^{d+1}$ to the system of $d+1$ Laurent polynomial equations 
    \begin{equation} \label{eq:systemdisc} A \, \psi_{A,b}(v) \, = \, \omega_0 + t \, \omega_1, \quad \det J_{A,b}(v) \, = \, 0.\end{equation}
    Let us denote this number by $\delta$. In fact, it is clear that if $(v,t)$ is a solution to these equations, then so is $(v',t)$ for any $v' \in {\rm pr}_\omega^{-1}(\omega_0 + t \, \omega_1)$. Hence, we have $\deg (\nabla_{A,b}) = (\deg {\rm pr}_\omega)^{-1} \cdot \delta$.  
    
    By Bernstein's theorem, the number $\delta$ is bounded by the mixed volume of the $d+1$ Newton polytopes of the equations \cite{bernstein1975number}. The Newton polytope of each of the first $d$ equations is contained in the $(d+1)$-dimensional polytope $\hat{P}_A = {\rm conv}(P_A \cup e_{d+1})$. Here we add the $(d+1)$-st standard basis vector to $P_A$ because of the term $t \, \omega_1$. The Newton polytope of the toric Jacobian determinant $\det J_{A,b}(v)$ is contained in the $d$-dimensional polytope $d \cdot P_A$. Hence, we have $\delta \leq {\rm MV}(\hat{P}_A, \ldots, \hat{P}_A, d \cdot P_A)$, where $\hat{P}_A$ is repeated $d$ times and ${\rm MV}(\cdot)$ denotes the mixed volume. By definition, this mixed volume equals the coefficient of $\lambda_0^d \lambda_1$ in ${\rm vol}(\lambda_0 \hat{P}_A + \lambda_1 dP_A)$. The Minkowski sum $\lambda_0 \hat{P}_A + \lambda_1 dP_A$ is obtained from the pyramid $(\lambda_0 + d\lambda_1) \hat{P}_A$ with volume $(\lambda_0 + d\lambda_1)^{d+1}(d+1)^{-1} {\rm vol}(P_A)$ by ``chopping off'' the top of the pyramid  $\lambda_0 e_{d+1} + d\lambda_1 \hat{P}_A$ with volume $(d\lambda_1)^{d+1}(d+1)^{-1} {\rm vol}(P_A)$. The coefficient standing with $\lambda_0^d\lambda_1$ in the difference of these expressions is $d \cdot d! {\rm vol}(P_A)$. We have shown that $\delta \leq d \cdot d!{\rm vol}(P_A)$, and since $\deg (\nabla_{A,b}) = (\deg {\rm pr}_\omega)^{-1} \cdot \delta$ the theorem follows. 
\end{proof}

\begin{example} \label{ex:experiment}
    The degree of the Lissajous discriminant can be computed by solving \eqref{eq:systemdisc} for generic $\omega_0, \omega_1$. There are $\delta$ solutions, whose image under ${\rm pr}_\omega$ consists of $\deg \nabla_{A,b}$ many points. To compute all $\delta$ solutions, we use the package \texttt{HomotopyContinuation.jl} (v2.15.0) \cite{HomotopyContinuation.jl} in Julia (v1.11). To compute the lattice volume $d! \, {\rm vol}(P_A)$, we use the function \texttt{lattice\_volume} from \texttt{Oscar.jl} (v1.1.2) \cite{OSCAR-book,OSCAR}. We apply these methods for matrices $A$ coming from cyclic and complete graphs. The results are reported in Table \ref{tab:Gdeg}. The column $d \cdot d! \, {\rm vol}(P_A)$ is computed using the command \texttt{d*lattice\_volume(convex\_hull(transpose([A -A])))}. By \cite[Theorem 14]{chen2018counting}, these numbers for $C_n$ are given by the sequence $n(n-1) \binom{n-1}{\lfloor \tfrac{1}{2} (n-1) \rfloor}$. The other entries can be reproduced using the following Julia code snippet, which uses standard~tools: 
\begin{minted}[fontsize=\footnotesize]{julia}
using HomotopyContinuation
A = ... # input a matrix A of full row rank, e.g., A = [1 0 -1; -1 1 0]
b = ... # input a vector b, e.g., b = [0;0;0]
d, n = size(A)
β = exp.(-im*b*pi/2); βinv = β.^(-1);
@var v[1:d] w[1:d] t # declare variables
Ad = [[maximum([aa,0]) for aa in A];-[minimum([aa,0]) for aa in A]]
y = [prod([v;w].^Ad[:,i]) for i = 1:n]
yinv = [prod([w;v].^Ad[:,i]) for i = 1:n]
ψ = [1//2*(β[i]*y[i] + βinv[i]*yinv[i]) for i = 1:n]
D = det(A*diagm([1//2*(β[i]*y[i]-βinv[i]*yinv[i]) for i = 1:n])*transpose(A))
myω0 = randn(ComplexF64,d); myω1 = randn(ComplexF64,d);
eqs = System([A*ψ - myω0+t*myω1; [v[i]*w[i]-1 for i = 1:d]; D], variables = [v;w;t])
R = HomotopyContinuation.solve(eqs)
δ = length(solutions(R))
degdisc = length(unique_points([[sol[end]] for sol in solutions(R)]))
degpr = δ/degdisc 
\end{minted}
\normalsize
The variables in line 6 are our variables $v_1, \ldots, v_d$ and the variable $t$ in Equation \ref{eq:systemdisc}. The variables $w_k$ play the role of the inverses of the $v_k$, as in Section \ref{sec:2}. The columns of the matrix \texttt{Ad} in line 7 are the concatenations of nonnegative integer vectors $a_j^+$ and $a_j^{-}$ satisfying $a_j = a_j^+ - a_j^-$. Lines 8-13 construct the system of equations \eqref{eq:systemdisc}, and line 14 solves it.  
We emphasize that numerical homotopy continuation methods are based on heuristics and do not provide a proof that the numbers in the table are correct. However, we expect they are. Our goal here is to exemplify the (non-)tightness of the bound in Theorem~\ref{thm:degdisc}. 
\begin{table}[h!]
\centering
\footnotesize
\begin{tabular}{c|c|cc|cc|cc}
$A$ & $d\cdot d!{\rm vol}(P_A)$ 
& \multicolumn{2}{c}{$b$ generic} 
& \multicolumn{2}{c}{$b=\mathbf{0}$} 
& \multicolumn{2}{c}{$b=\mathbf{1}$} \\
& & $\deg(\nabla_{A,b})$ & $\deg(\mathrm{pr}_\omega)$
  & $\deg(\nabla_{A,\bf 0})$ & $\deg(\mathrm{pr}_\omega)$
  & $\deg(\nabla_{A,\bf 1})$ & $\deg(\mathrm{pr}_\omega)$ \\
\hline
$A(C_3)$ & $12$ & $12 $ & $1$ & $6$ & $2$ & $12$ & $1$ \\
$A(C_4)$ & $36$ & $36$ & $1$& $\bf 12$ & $2$ & $\bf 12$ & $2$\\
$A(C_5)$ & $120$ & $120$ & $1$ & $60$ & $2$& $120$ & $1$\\
$A(C_6)$ & $300$ & $300$  & $1$ & $\bf 130$  & $2$ & $150$ & $2$\\
$A(C_7)$ & $840$ & $ 840$ & $1$ & $420$ & $2$ &  $840$ & $1$\\
$A(C_8)$ & $1960$ & $1960$ & $1$ & $\bf 910$ & $2$  & $\bf 910$ & $2$ \\ \hline 
$A(K_4)$ & $60$ & $60$ & $1$ & $\bf 26$ & $2$ & $\bf 48$ & $1$ \\
$A(K_5)$ & $280$ & $280$ & $1$ & $\bf 90$ & $2$ & $\bf 140$ & $1$
\\
$A(K_6)$ & $1260$ & $ 1260$ & $1$ & $\bf 276$ & $2$ & $\bf 360$ & $1$
\end{tabular}
\caption{$\deg(\nabla_{A,b})$ for different reduced incidence matrices $A = A(G)$. Numbers in bold indicate that the upper bound from Theorem \ref{thm:degdisc} is \emph{not} attained.}
\label{tab:Gdeg}
\end{table}
\end{example}

Theorem \ref{thm:degdisc} and Example \ref{ex:experiment} show that Lissajous discriminants may have large degrees, which makes them challenging to compute. In some cases, the polynomial $\Delta_{A,b}$ exhibits some symmetries, and this can be exploited in the computation. Here is an example.

\begin{proposition}\label{prop:evendegdiscriminant}
    The discriminant $\Delta_{A,\bf 1}$ satisfies $\Delta_{A,\bf 1}(\omega) = \alpha \Delta_{A, \bf 1}(-\omega)$, where $\alpha = \pm 1$. In particular, the degree of each monomial appearing in $\Delta_{A,\bf 1}$ is even if $\alpha = 1$, or odd if~$\alpha = -1$. 
\end{proposition}
\begin{proof}
The Lissajous discriminant is stable under changing the sign of $\omega$. Indeed, we have 
\[ A \psi_{A,\bf 1}(v) = \omega \, \text{ and } \, \det J_{A,\bf 1}(v) = 0 \, \, \, \Longleftrightarrow \, \, \, A \psi_{A,\bf 1}(v^{-1}) = -\omega \, \text{ and } \, \det J_{A,\bf 1}(v^{-1}) = 0.\]
This implies that $\Delta_{A,\bf 1}(\omega)=\alpha \, \Delta_{A,\bf 1}(-\omega)$ for some $\alpha \in \mathbb{C}^*$. Changing signs twice, we find $\Delta_{A,\bf 1}(\omega)=\alpha \, \Delta_{A,\bf 1}(-\omega)=\alpha^2 \, \Delta_{A,\bf 1}(\omega)$, hence $\alpha=\pm1$.
Finally, write $\Delta_{A,\bf 1}=\Delta_{e} + \Delta_{o}$, where $\Delta_{e}$ is the sum of the monomials of $\Delta_{A,\bf 1}$ with even degree, and $\Delta_{o}$ contains those with odd degree. If $\alpha=1$, then $2\Delta_{o}=\Delta_{A,\bf 1}(\omega)-\Delta_{A,\bf 1}(-\omega)=0$. Similarly, if $\alpha = -1$, then $\Delta_e = 0$.
\end{proof}
 
If $A$ comes from a graph $G$, as in Section \ref{sec:4}, then the Lissajous discriminant respects the symmetries of $G$. We shall now make this more precise. 
Recall that an automorphism of $G$ is a permutation of its vertices which preserves adjacency. These constitute the automorphism group ${\rm Aut}(G)$, a subgroup of the symmetric group of order $n!$. For instance, the automorphism group of the complete graph $K_n$ is ${\rm Aut}(K_n) = S_n$, the full symmetric group. The automorphism group of the cycle graph $C_n$ is the dihedral group of order $2n$. 

Previously, we have worked with the reduced incidence matrix $A(G) \in \mathbb{Z}^{d \times n}$, which has full rank $d$. To describe the symmetries of the discriminant, it is more natural to work with the full incidence matrix $A_G \in \mathbb{Z}^{(d+1) \times n}$ of rank $d$ and consider the equivalent equations 
\[ A_G \, \psi_{A_G,\bf 1}(v) = \omega, \, \, v_{d+1} = 1 \quad \text{and} \quad  {\rm rank}(J_{A_G,b}(v)) < d. \]
Since the rows of $A_G$ sum to zero, these equations imply that $\omega_1 + \cdots + \omega_d + \omega_{d+1} = 0$. The Lissajous discriminant $\nabla_{A_G,\bf 1}$ is (expected to be) a hypersurface \emph{inside the hyperplane} $\omega_1 + \cdots + \omega_d + \omega_{d+1} = 0$. Its projection onto the first $d$ coordinates is $\nabla_{A(G),\bf 1}$. 

\begin{proposition} \label{prop:symmetry}
    Let $G$ be a graph with $m = d+1$ vertices and let $\sigma \in {\rm Aut}(G)$. We have $\omega \in \nabla_{A_G, \bf 1} \subset \mathbb{C}^{d+1}$ if and only if $\sigma(\omega) \in \nabla_{A_G,\bf 1}$, where $\sigma$ acts on $\omega$ by permuting coordinates.
\end{proposition}

\begin{proof}
    The group ${\rm Aut}(G)$ acts on $(\mathbb{C}^*)^{d+1}$ and on $\mathbb{C}^{d+1}$ by permuting coordinates. The proposition follows from the observation that the map $f_G: (\mathbb{C}^*)^{d+1} \rightarrow \mathbb{C}^{d+1}$ given by $v \mapsto A_G \, \psi_{A_G,\bf 1}(v)$ is equivariant with respect to these actions, meaning that $f_G(\sigma(v)) = \sigma(f_G(v))$, $\sigma \in {\rm Aut}(G)$. 
    To prove this, write the $k$-th coordinate of $f_G(v)$ as 
    \[
    f_G(v)_k \, = \, \sum_{(k,j)>0} \tfrac{1}{2i}(v_k v_j^{-1}-v_k^{-1} v_j)-\sum_{(k,j)<0} \tfrac{1}{2i}(v_j v_k^{-1}-v_j^{-1} v_k), \]
    where $(k,j)>0$ denotes a sum over all edges $(k,j)$ oriented from $k$ to $j$, and similarly for $(k,j) < 0 $. This simplifies to 
    $f_G(v)_k=\sum_{(k,j)\in E_k} \tfrac{1}{2i}(v_k v_j^{-1}-v_k^{-1} v_j)$, where $E_k$ is the set of all edges adjacent to vertex $k$.
    Now the equality $f_G(\sigma (v))_k = f_G(v)_{\sigma(k)}$ is clear from
     \[
     \sum_{(k,j)\in E_k} \tfrac{1}{2i}(v_{\sigma(k)} {v_{\sigma(j)}}^{-1}-{v_{\sigma(k)}}^{-1} v_{\sigma(j)}) = \sum_{(\sigma(k),\sigma(j))\in E_{\sigma(k)}} \tfrac{1}{2i}(v_{\sigma(k)} {v_{\sigma(j)}}^{-1}-{v_{\sigma(k)}}^{-1} v_{\sigma(j)}),
     \]
     since $(k,j)\in E_k$ if and only if $(\sigma(k),\sigma(j))\in E_{\sigma(k)}$, by definition of ${\rm Aut}(G)$.
\end{proof}

\begin{example}
    Expanding $\Delta_{A({C_3}),\bf 1}$ from Example \ref{ex:discrunning} as a polynomial in $\omega_1$ and $\omega_2$, we obtain 41 terms, all of even degree. The fact that $\Delta_{A_{C_3},\bf 1}(\omega_1,\omega_2,\omega_3)$ can be expressed in terms of elementary symmetric polynomials mirrors the fact that ${\rm Aut}(C_3) = S_3$.
\end{example}

\vspace{-0.5cm}

\footnotesize

\bibliographystyle{abbrv}
\bibliography{references}

\begin{thebibliography}{10}

\bibitem{bel2024chebyshev}
Z.~Bel-Afia, C.~Meroni, and S.~Telen.
\newblock Chebyshev varieties.
\newblock {\em Mathematics of Computation}, 95(4125), 2025.

\bibitem{bernal2023machine}
E.~A. Bernal, J.~D. Hauenstein, D.~Mehta, M.~H. Regan, and T.~Tang.
\newblock Machine learning the real discriminant locus.
\newblock {\em Journal of Symbolic Computation}, 115:409--426, 2023.

\bibitem{bernstein1975number}
D.~N. Bernstein.
\newblock The number of roots of a system of equations.
\newblock {\em Functional Analysis and Its Applications}, 9(3):183--185, 1975.

\bibitem{HomotopyContinuation.jl}
P.~Breiding and S.~Timme.
\newblock {H}omotopy{C}ontinuation.jl: {A} {P}ackage for {H}omotopy
  {C}ontinuation in {J}ulia.
\newblock In {\em International Congress on Mathematical Software}, pages
  458--465. Springer, 2018.

\bibitem{chen2018counting}
T.~Chen, R.~Davis, and D.~Mehta.
\newblock Counting equilibria of the {K}uramoto model using birationally
  invariant intersection index.
\newblock {\em SIAM Journal on Applied Algebra and Geometry}, 2(4):489--507,
  2018.

\bibitem{CoxLittleSchenck2011}
D.~A. Cox, J.~B. Little, and H.~K. Schenck.
\newblock {\em Toric Varieties}, volume 124 of {\em Graduate Studies in
  Mathematics}.
\newblock American Mathematical Society, 2011.

\bibitem{de2012central}
J.~A. De~Loera, B.~Sturmfels, and C.~Vinzant.
\newblock The central curve in linear programming.
\newblock {\em Foundations of Computational Mathematics}, 12(4):509--540, 2012.

\bibitem{OSCAR-book}
W.~Decker, C.~Eder, C.~Fieker, M.~Horn, and M.~Joswig, editors.
\newblock {\em The {C}omputer {A}lgebra {S}ystem {OSCAR}: {A}lgorithms and
  {E}xamples}, volume~32 of {\em Algorithms and {C}omputation in
  {M}athematics}.
\newblock Springer, 1 edition, 2025.

\bibitem{dorfler2014synchronization}
F.~D{\"o}rfler and F.~Bullo.
\newblock Synchronization in complex networks of phase oscillators: a survey.
\newblock {\em Automatica}, 50(6):1539--1564, 2014.

\bibitem{MourrainElkadi2007ResolutionSystemesPolynomiaux}
M.~Elkadi and B.~Mourrain.
\newblock {\em Introduction {\`a} la r{\'e}solution des syst{\`e}mes
  polynomiaux}, volume~59 of {\em Math. Appl. (Berl.)}.
\newblock Berlin: Springer, 2007.

\bibitem{greenslade1993all}
T.~B. Greenslade.
\newblock All about {L}issajous figures.
\newblock {\em Physics Teacher}, 31(6):364--70, 1993.

\bibitem{harrington2023kuramoto}
H.~A. Harrington, H.~Schenck, and M.~E. Stillman.
\newblock Algebraic aspects of homogeneous {K}uramoto oscillators.
\newblock {\em Mathematics of Computation}, 95(358):1023--1047, 2025.

\bibitem{kouchnirenko1976polyedres}
A.~G. Kouchnirenko.
\newblock Poly{\`e}dres de {N}ewton et nombres de {M}ilnor.
\newblock {\em Inventiones mathematicae}, 32(1):1--31, 1976.

\bibitem{kuramoto1975self}
Y.~Kuramoto.
\newblock Self-entrainment of a population of coupled nonlinear oscillators.
\newblock In H.~Araki, editor, {\em International Symposium on Mathematical
  Problems in Theoretical Physics}, volume~39 of {\em Lecture Notes in
  Physics}, pages 420--422. Springer, 1975.

\bibitem{Laurent1997}
M.~Laurent.
\newblock The real positive semidefinite completion problem for series-parallel
  graphs.
\newblock {\em Linear Algebra and its Applications}, 252(1--3):347--366, 1997.

\bibitem{ling2019landscape}
S.~Ling, R.~Xu, and A.~S. Bandeira.
\newblock On the landscape of synchronization networks: A perspective from
  nonconvex optimization.
\newblock {\em SIAM Journal on Optimization}, 29(3):1879--1907, 2019.

\bibitem{Lubin2023}
M.~Lubin, O.~Dowson, J.~{Dias Garcia}, J.~Huchette, B.~Legat, and J.~P. Vielma.
\newblock {JuMP} 1.0: {R}ecent improvements to a modeling language for
  mathematical optimization.
\newblock {\em Mathematical Programming Computation}, 2023.

\bibitem{mehta2015algebraic}
D.~Mehta, N.~S. Daleo, F.~D{\"o}rfler, and J.~D. Hauenstein.
\newblock Algebraic geometrization of the {K}uramoto model: Equilibria and
  stability analysis.
\newblock {\em Chaos: An Interdisciplinary Journal of Nonlinear Science},
  25(5), 2015.

\bibitem{mumford1999red}
D.~Mumford.
\newblock {\em The Red Book of Varieties and Schemes}, volume 1358 of {\em
  Lecture Notes in Mathematics}.
\newblock Springer-Verlag, 2nd, expanded edition, 1999.
\newblock Includes the Michigan Lectures (1974) on Curves and their Jacobians.

\bibitem{OSCAR}
{OSCAR} -- {O}pen {S}ource {C}omputer {A}lgebra {R}esearch system, version
  1.1.2, 2025.

\bibitem{pavlov2025santalo}
D.~Pavlov and S.~Telen.
\newblock Santal{\'o} geometry of convex polytopes.
\newblock {\em SIAM Journal on Applied Algebra and Geometry}, 9(1):58--82,
  2025.

\bibitem{sanyal2013entropic}
R.~Sanyal, B.~Sturmfels, and C.~Vinzant.
\newblock The entropic discriminant.
\newblock {\em Adv. Math.}, 244:678--707, 2013.

\bibitem{SolusUhlerYoshida2016}
L.~Solus, C.~Uhler, and R.~Yoshida.
\newblock Extremal positive semidefinite matrices for graphs without {$K_5$}
  minors.
\newblock {\em Linear Algebra and its Applications}, 509:247--275, 2016.

\bibitem{sturmfels2024toric}
B.~Sturmfels, S.~Telen, F.-X. Vialard, and M.~von Renesse.
\newblock Toric geometry of entropic regularization.
\newblock {\em Journal of Symbolic Computation}, 120:102221, 2024.

\bibitem{sturmfels2010multivariate}
B.~Sturmfels and C.~Uhler.
\newblock Multivariate {G}aussians, semidefinite matrix completion, and convex
  algebraic geometry.
\newblock {\em Annals of the Institute of Statistical Mathematics},
  62(4):603--638, 2010.

\bibitem{Telen2025MonomialMaps}
S.~Telen.
\newblock Chapter 1: Monomial maps and toric ideals.
\newblock In {\em Applied Toric Geometry}. 2025.
\newblock Book in progress. Available at
  \url{https://sites.google.com/view/simontelen/teaching}.

\bibitem{vinzant2014spectrahedron}
C.~Vinzant.
\newblock What is... a spectrahedron?
\newblock {\em Notices of the AMS}, 61(5):492--494, 2014.

\end{thebibliography}

\noindent{\bf Authors' addresses:}
\smallskip

\noindent Francesco Maria Mascarin, MPI-MiS Leipzig
\hfill {\tt francesco.mascarin@mis.mpg.de}

\noindent Simon Telen, MPI-MiS Leipzig
\hfill {\tt simon.telen@mis.mpg.de}

\end{document}